\documentclass[12pt,reqno]{amsart}
\title[HDG eigenproblem \& hybridization]%
{Spectral approximations by the HDG method}

\thanks{{\LL This work was partially supported by NSF through grants
    DMS-1211635, DMS-1318916, and the CAREER award DMS-0847241, by an
    Alfred P. Sloan Research Fellowship, and by AFOSR under grant
    FA9550-12-0357.}}

\author{J. Gopalakrishnan}
\address{Portland State University, PO Box 751, Portland,
OR 97207-0751}
\email{gjay@pdx.edu}

\author{F. Li}
\address{Department of Mathematical Sciences,
  Rensselaer Polytechnic Institute, Troy,
  NY~12180.}
\email{lif@rpi.edu}

\author{N.-C. Nguyen}
\address{Department of Aeronautics and Astronautics,
  Massachusetts Institute of Technology, Cambridge, MA02139, USA.}
\email{cuongng@mit.edu}

\author{J. Peraire}
\address{Department of Aeronautics and Astronautics,
  Massachusetts Institute of Technology, Cambridge, MA02139, USA.}
\email{peraire@mit.edu}

\newcommand{\grad}{\vec{\nabla}}
\newcommand{\Shk}{S_h^{(\kappa)}}
\newcommand{\Shl}{S_h^{(\lambda_h)}}
\theoremstyle{plain}
\newtheorem{theorem}{Theorem}[section]
\newtheorem{lemma}[theorem]{Lemma}%[section]
\newtheorem{corollary}[theorem]{Corollary}%[section]

\theoremstyle{remark}

\usepackage{color}
\definecolor{bbchange}{rgb}{0,0,0}
\definecolor{ccchange}{rgb}{0.8,0,0}
\newcommand\bbchange[1]{{#1}}

\usepackage{graphicx,amssymb,amsmath,amsbsy,mathrsfs,eucal,amsfonts,enumitem,subfigure}

\def\d{\partial}

\newcommand{\Om}{\ensuremath{\Omega}}

\newcommand{\et}{\ensuremath{\tilde{\eta}}}

\newcommand{\lt}{\ensuremath{\tilde{\lambda}}}

\newcommand{\vG}{\varGamma}

\newcommand{\vqg}{\vq^{\;g}}
\newcommand{\vqf}{\vq^{\,f}}

\newcommand{\diam}{\ensuremath{\mathop\mathrm{diam}}}
\newcommand{\dist}{\ensuremath{\mathop\mathrm{dist}}}

\newcommand{\Poincare}{Poincar{\'{e}}}

\newcommand{\Hdiv}[1]{H(\mathrm{div},{#1})}
\newcommand{\qhat}{\widehat{q}}

\newcommand{\Th}{{\mathscr{T}_h}}
\newcommand{\oh}{{\Th}}
\newcommand{\doh}{{\partial \!\Th}}

\newcommand{\taum}{\tau_K^{\mathrm{max}}}

\newcommand{\ip}[1]{\langle {#1} \rangle}

\newcommand{\ii}{\ensuremath{\hat \i}}

\newcommand{\pr}{\varPi}
\newcommand{\pv}{\varPi_h^{\scriptscriptstyle{V}}}
\newcommand{\pw}{\varPi_h^{\scriptscriptstyle{W}}}

\newcommand{\Uw}{\mathcal{U}_{\!\scriptscriptstyle{W}}}
\newcommand{\Qw}{\vec{\mathcal{Q}}_{\!\scriptscriptstyle{W}}}

\newcommand{\vq}{{\vec{q}}}

\newcommand{\vr}{{\vec{r}}}
\newcommand{\vn}{{\vec{n}}}
\newcommand{\vv}{{\vec{v}}}
\newcommand{\vx}{{\vec{x}}}

\newcommand{\vw}{{\vec{w}}}

\newcommand{\dive}{\mathop{\nabla}\cdot\,}

\newcommand{\Forall}{\quad\text{for all }}

\newcommand{\RRR}{\mathbb{R}}

\newcommand{\LU}{\mathcal{U}}

\newcommand{\QQ}{\vec{\mathcal{Q}}}
\newcommand{\UU}{\mathcal{U}}

\newcommand{\LL}{}

\begin{document}

\begin{abstract}
  We consider the numerical approximation of the spectrum of a
  second-order elliptic eigenvalue problem by the hybridizable
  discontinuous Galerkin (HDG) method. We show for problems with
  smooth eigenfunctions that the approximate eigenvalues and
  eigenfunctions converge at the rate $2k+1$
  and $k+1$, respectively. Here $k$ is the degree of the polynomials
  used to approximate the solution, its flux, and the numerical
  traces. Our numerical studies show that a Rayleigh quotient-like
  formula applied to certain locally postprocessed approximations can
  yield eigenvalues that converge faster at the rate $2k + 2$ for 
  the HDG method as well as for 
  the Brezzi-Douglas-Marini (BDM) method. We also
  derive and study a condensed nonlinear eigenproblem for the
  numerical traces obtained by eliminating all the other variables.
\end{abstract}


\maketitle

\section{Introduction}         \label{sec:intro}

We study the HDG (hybridized discontinuous Galerkin) approximation
to the following eigenproblem: Find eigenvalues $\lambda$ in $\RRR$
and corresponding nontrivial eigenfunctions $u$ satisfying
\begin{equation}
  \label{eq:lambda}
-\nabla\cdot (\alpha\grad u) = \lambda\,u\quad\mbox{ in }\Omega,
\qquad u=0\quad\mbox{ on }\partial\Omega.
%  (\alpha \grad u, \grad v) = \lambda (u,v) \qquad\forall v \in H_0^1(\Omega),
\end{equation}
{\LL Assumptions are placed on $\alpha$ and $\Omega$ in Section~\ref{sec:hdg}.}
Several HDG discretizations were introduced
in~\cite{CockbGopalLazar09} to discretize corresponding source
problems.  The purpose of this paper is to study the application of
one such method to eigenproblems. The particular HDG method considered
here (referred to as the H-LDG method in~\cite{CockbGopalLazar09}, but
simply as the ``HDG method" in this paper) is chosen for our study
because we now have a fairly complete theoretical understanding of its
application to the source problem.

Two well-known advantages of the HDG method, when applied to source
problems, are reduced system size and flexible stabilization.  The
latter arises due to a transparent stabilization mechanism in the
definition of numerical flux. It allows one to use, for example,
polynomials of the same degree $k$ to approximate the solution $u$ and
components of its flux~$\vq=-\alpha \grad u$. While use of these
spaces would have resulted in an unstable mixed method, the resulting
HDG method is stable, and provides optimal order approximations for
all variables. To discuss the former, namely the advantage of reduced
system size, recall the process of static condensation, which, for
source problems, removes all interior variables to yield a
``condensed'' system for inter-element variables.  The HDG condensed
system, when compared to condensed systems from other DG methods, is
attractive because of its smaller size and favorable sparsity
patterns. One of the questions we address in this paper is whether
such condensed systems are useful in eigenproblems. Note that the
condensation process, by reducing the system size, also reduces the
size of the spectrum.  How much of the spectrum can be approximated
despite this reduction is a question answered in
Section~\ref{sec:hybrid}. There, we also derive the nonlinear equation
that needs to be solved in the reduced dimensions to compute the
spectral approximations.

Apart from these results on the condensed eigenproblem, we prove
convergence results for the HDG eigenproblem in
Sections~\ref{sec:convergence} and~\ref{sec:rate}.  We show that the
approximate eigenvalues obtained by the HDG method exhibit no spectral
pollution. They converge to the exact eigenvalues at the rate
$O(h^{2k+1})$, under suitable regularity assumptions, if we use
polynomials of degree at most $k$ for all the HDG variables.  We also
show, under similar assumptions, that the gap between the
corresponding discrete eigenspace and the exact eigenspace in
$L^2(\Omega)$ converges at $O(h^{k+1})$.  Roughly speaking, this
shows that the rate of convergence of eigenfunctions in
$L^2(\Omega)$ is optimal for the HDG method.

These results compare favorably with many other DG eigenvalue
approximations~\cite{AntonBuffaPerug06,GianiHall2012}. The unified
presentation of~\cite{AntonBuffaPerug06} facilitates comparison.  They
show that many traditional Hermitian DG methods approximate
eigenvalues at the rate $O(h^{2k})$.  For non-Hermitian DG methods,
they find that the eigenvalue convergence rate is even poorer, in
general no better than $O(h^k)$.  The HDG method, which
can be considered to fall in the Hermitian class, yields
eigenvalues that converge faster,
when compared to both the Hermitian and the non-Hermitian DG methods considered
in~\cite{AntonBuffaPerug06}.
However, let us note that the convergence rate of HDG eigenvalues (or,
for that matter, any DG eigenvalues) do not compare favorably with the
$O(h^{2k+2})$ convergence rate of the mixed hybridized Raviart-Thomas
(HRT) method~\cite{BoffiBrezzGasta00,hrt:2010,MerciOsborRappa81}.  Our
analysis also points to other differences when comparing the HDG and
HRT eigenproblems. For example, comparing
Theorem~\ref{thm:reduction:eig} below with
\cite[Theorem~3.2]{hrt:2010}, we note that the extent of the spectrum
recovered by the condensed system may be shorter (up to $O(1/h)$) for
the HDG method in comparison to the HRT method (which is up to
$O(1/h^2)$).

The method of convergence analysis in this paper is motivated by the
many early works that developed abstract approaches to analyze
approximation of eigenproblems~\cite{BrambOsbor73,
  DesclNassiRappa78a,MerciOsborRappa81,Osbor75}, and in particular,
the application of the abstract theory to DG methods
in~\cite{AntonBuffaPerug06}. The critical new tool that helps push the
analysis forward in the HDG case, and yield better convergence rates
than~\cite{AntonBuffaPerug06}, is the projection operator
of~\cite{CockbGopalSayas10}. The projection allows our analysis of HDG
eigenvalue errors to proceed along the lines of similar analyses for
mixed methods~\cite{BoffiBrezzGasta00}. A few important differences
arise due to the fact that the HDG projection possesses only a weak
analogue of a well-known commutativity property. The analysis in the second
aspect of this study, involving the condensed system, is motivated by
our previous such analysis~\cite{hrt:2010} for the HRT method.

Reports of extensive numerical experiments are provided in
Section~\ref{sec:numer}. Of particular interest is a local and
inexpensive postprocessing presented there. The postprocessed
eigenvalues seem to converge at $O(h^{2k+2})$-rate thus making the HDG
method competitive with the mixed method.  The intuition behind the
construction of this postprocessing is inherited from our previous
experience with the mixed method~\cite{hrt:2010}, where we proved that
the postprocessed eigenfunctions have better convergence
rates. However, the argument used there to provide a rigorous proof
does not seem to extend to the HDG method.

In the next section, we introduce the HDG eigenproblem and essential
notations used throughout the paper. In Section~\ref{sec:convergence},
we show that there is no pollution of the spectrum when approximating
it by the HDG method. This result is improved in
Section~\ref{sec:rate}, where we establish (in
Theorem~\ref{thm:ewrate}) the convergence rate of HDG
eigenvalues. Section~\ref{sec:hybrid} investigates the condensation
(or hybridization) of the HDG eigenproblem to obtain a smaller condensed
nonlinear eigenproblem for the interface variable otherwise known as
the numerical trace. The condensed nonlinear eigenproblem is given in
Theorem~\ref{thm:reduction:eig} and a closely related linear
eigenproblem is investigated in Theorem~\ref{thm:init}.
{\LL
We conclude
in Section~\ref{sec:numer} with the results of our numerical studies
and a brief discussion of the eigenproblem using
 mixed-degree
polynomial spaces.
}

\section{The HDG source and eigenvalue problems}  \label{sec:hdg}

Consider the Dirichlet boundary value problem (rewritten as a first
order system) of finding $\vqf \in \Hdiv\Omega$ and $u^f\in
L^2(\Omega)$, given any ``source'' $f$ in $L^2(\Omega)$, such that
\begin{subequations}
\label{eq:theproblem}
\begin{align}
\label{eq:q+grad(u)=0}
\vqf + \alpha\, \grad u^f
                        & =0      && \quad \text { on } \Omega,\\
\label{eq:div(q)+du=f}
\dive\, \vqf             & = f     && \quad \text { on } \Omega,\\
\label{eq:u=g}
u\bbchange{^f}                       & = 0     && \quad \text{ on } \d \Omega.
\end{align}
\end{subequations}
All functions, unless explicitly stated otherwise, are real-valued in
this paper. Throughout, $\Omega \subset \RRR^n$ is a polyhedral domain ($n
\ge 2$), $\alpha:\Omega \to \RRR^{n\times n}$ denotes a variable
matrix valued coefficient, which we assume to be symmetric and
positive definite at all points in $\Omega$. {\LL We assume that there is a fixed constant that bounds the norms of
$\alpha$ and $c=\alpha^{-1}$ for all $x \in\Omega$, dependence
on which is not tracked in the estimates of this paper. }
To facilitate \bbchange{our %later
analysis}, we introduce notation for the ``solution operator'' $T:
L^2(\Omega)\to L^2(\Omega)$, which is defined simply by
\begin{equation}
  \label{eq:T}
T f = u^f.
\end{equation}
It is well known that $T$ is compact and self-adjoint. Its spectrum,
denoted by $\sigma(T)$, consists of isolated points on the positive
real line accumulating at zero.  Clearly, there is a one to one
correspondence between the eigenvalues of~\eqref{eq:lambda} and those
of $T$. Indeed, $\mu$ is an eigenvalue of $T$ if and only if $\mu =
1/\lambda$ for some $\lambda$ satisfying~\eqref{eq:lambda}.

\subsection{The source problem}   \label{ssec:source}

The HDG method provides an approximation $T_h$ to $T$. To understand
this approximation, we first describe the HDG source problem and
introduce known results we shall use later. Afterward, we will present
the HDG eigenvalue problem.

The HDG method defines a scalar approximation $u_h$ to $u$ and a
vector approximation $\vq_h$ to $\vq$\, in the following spaces,
respectively:
\begin{align}
  \label{eq:W}
  W_h &= \{ w: \text{ for every mesh element }K, w|_K \in P_k(K)\},
  \\  \label{eq:V}
  V_h &=  \{ \vv: \text{ for every mesh element }K, \vv|_K \in P_k(K)^n\}.
\end{align}
Note that functions in these spaces need not be continuous across
element interfaces.  Above and elsewhere, we use $P_k(D)$ to denote
the space of polynomials of degree at most $k\ge 0 $ on some
domain~$D$. The subscript $h$ denotes the mesh size defined as the
maximum of the diameters of all mesh elements.

For any (scalar or vector) function $q$ in $V_h$ or $W_h$, the trace
$q|_F$ is, in general, a double-valued function on any interior mesh
face $F = \d K^+ \cap \d K^-$ shared by the mesh elements $K^+$ and
$K^-$. Its two branches, denoted by $[q]_{K^+}$ and $[q]_{K^-}$, are
defined by $[q]_{K^\pm}(\vec{x})=
\lim_{\epsilon\downarrow0}\,q(\vec{x}-\epsilon\,[\vec{n}]_{K^\pm})$ for
all $\vx$ in $F$. Here and elsewhere, $\vn$ denotes the double-valued
function of unit normals on the element interfaces: on a face
$F\subseteq \d K$, its branch $[\vn]_K$ equals the unit  normal
on $\d K$ pointing outward from~$K$. For functions $u$ and $v$ in $L^2(D)$, we write $(u,v)_D =
\int_D u v \; dx$ whenever $D$ is a domain of $\RRR^n$, and
$\ip{u,v}_D = \int_D u v \; dx$ whenever $D$ is a domain of
$\RRR^{n-1}$. To simplify the notation,  define
\[
(v,w)_{\oh}=\sum_{K \in \Th} (v,w)_K
\qquad\text{and}\qquad
\ip{v,w}_{\doh} =
\sum_{K \in \Th}
\ip{ v,w}_{\d K},
\]
where in the latter, we understand that for double-valued $v$ and $w$,
the integral $\ip{v,w}_{\d K}$ is computed using the branches $[v]_K$
and $[w]_K$ from $K$.
For vector functions $\vv$ and $\vw$, the notations are similarly
defined with the integrand being the dot product $\vv \cdot \vw$.

In addition to the spaces  $V_h$ and $W_h$ introduced above, our method also
uses one other discrete space~$M_h$, consisting of functions defined
on the mesh faces (or mesh edges  if $n=2$), namely
\begin{equation}
  \label{eq:M}
  M_h = \{ \mu: \text{ for every mesh face } F, \;\mu|_F \in
  P_k(F), \text{ and if } F\subseteq \d\Omega, \; \mu|_F=0\}.
\end{equation}
The HDG method defines the approximate solution $u_h^f$, the
approximate flux $\vq_h^f$,
and the numerical trace
$\eta_h^f$, as the functions in $W_h$, $V_h,$ and $M_h$,
respectively,
satisfying
\begin{subequations}
\label{method}
\begin{alignat}{2}
\label{eq:method-a}
 ( c\, \vq_h^f,\vr)_{\oh} -
 ( u_h^f,\dive \vr)_{\oh} +   \ip{\eta_h^f,\vr\cdot\vn}_{\doh}
 &  = 0,
 &&\Forall \vr\in V_h,
 \\
\label{eq:method-b}
 -( \vq_h^f, \grad w)_{\oh} +
 \ip{ \qhat_h^f\cdot\vn, w}_{\doh}
 & =
 (f, w)_{\oh}
 &&
 \Forall w \in W_h,\\
\label{eq:method-c}
 \ip{\mu, \qhat_h^f \cdot \vn }_{\doh}
   & = 0
   && \Forall \mu \in M_h,
\end{alignat}
\end{subequations}
where
%$c=\alpha^{-1}$ and
$\qhat_h^f$ is a  double-valued vector function on mesh
interfaces defined by
\begin{equation}
  \label{eq:qhatf}
\qhat_h^f = \vq_h^f + \tau \big( u_h^f - \eta_h^f \big) \vn.
\end{equation}
Note that this defines all branches, i.e., on the boundary $\d K$ of
every mesh element $K$, the value of the branch of $\qhat_h^f$ from
$K$ is $[\qhat_h^f]_K = [\vq_h^f]_K + [\tau]_K \big( [u_h^f]_K -
\eta_h^f \big)[ \vn]_K.  $ Here $\tau$ is a non-negative penalty
function, which is also double-valued on the element interfaces and
$[\tau]_K$ above denotes the branch of $\tau$-values from~$K$. For
simplicity, we assume that any branch of $\tau$ is a constant function
on each mesh edge. It is proved in~\cite{CockbGopalLazar09} that the
system~\eqref{method} is uniquely solvable if {\LL $[\tau]_K$ }
is positive on
at least one face of~$K$ for every element~$K$.  This unique
solvability result is assumed throughout this paper. Given any $f$
in $L^2(\Omega)$, the component $u_h^f$ of the unique solution
of~\eqref{method} is used to define the discrete version of the operator~$T$
in~\eqref{eq:T}, namely
\begin{equation}
  \label{eq:Th}
  T_h f = u_h^f.
\end{equation}

We will need a projection $\pr_h(\vq,u)$, into the product space $V_h
\times W_h$, originally designed in~\cite{CockbGopalSayas10}. Its
domain is a subspace of $\Hdiv\Omega \times L^2(\Omega)$ consisting of
sufficiently regular functions, e.g., $\Hdiv\Omega \cap H^s(\Omega)^n
\times H^s(\Omega)$ for $s>1/2.$ When its components need to be
identified, we also write $ \pr_h (\vq, u)$ as $ (\pv \vq, \pw u) $
where $\pv \vq$ and $\pw u$ are the components of the projection in
$V_h$ and $W_h$, respectively. (Despite this notation, note that $\pv
\vq$ depends not just on $\vq$, but rather on both $\vq$ and $u$. The
same applies for $\pw u$.)  The components are defined by
\begin{subequations}
  \label{eq:proj}
  \begin{align}
    \label{eq:proj1}
    (\pv \vq, \vr)_K
    &= ( \vq, \vr)_K
    && \Forall \vr \in P_{k-1}(K)^n,
    \\
    \label{eq:proj2}
    (\pw u, w )_K
    &= ( u, w )_K
    && \Forall w \in P_{k-1}(K),
    \\
    \label{eq:proj3}
    \ip{\pv \vq \cdot \vn + \tau \pw u,\mu }_F
    &= \ip{\vq\cdot\vn + \tau  u,\mu }_F
    && \Forall \mu \in P_k(F),
  \end{align}
\end{subequations}
for all faces $F$ of the simplex $K$.  % If $k=0$, then~\eqref{eq:proj1}
% and~\eqref{eq:proj2} are vacuous and $\pr_h$ is defined solely
% by~\eqref{eq:proj3}.
% We omit the definition~\cite[eq.~(2.1)]{CockbGopalSayas10}
% and other details of the projection, but
Let $s_u, s_q \in (1/2,k+1]$. We recall the following approximation
property, proved in~\cite[Theorem~2.1]{CockbGopalSayas10} for integer
values of $s_u, s_q$, and extended to remaining values of $s_u, s_q \in
(1/2,k+1]$ in~\cite{CockbDuboiGopal10}:

For all $( \vq, u) \in \Hdiv\Omega \cap H^{s_q}(\Omega)^n
\times H^{s_u}(\Omega)$,
  \begin{subequations}
      \label{eq:PIapprx}
    \begin{alignat}{1}
      \label{eq:PIapprx-a}
      \| \pv \vq - \vq\, \|_{\LL L^2(K)}\le
      &\;C\,h_K^{s_q}\,|\vq|_{H^{s_q}(K)}
       +C\, h_K^{s_u}\,{\tau_K^{*}}\,|u|_{H^{s_u}(K)}
      \\
      \label{eq:PIapprx-b}
      \| \pw  u - u \|_{\LL L^2(K)} \le
      &\;C\,h^{s_u}_K\,|u|_{H^{s_u}(K)}
      +C\,\frac{h^{s_q}_K}{\tau_K^{\max}}\,|\vq|_{H^{s_q}(K)}.
    \end{alignat}
  \end{subequations}
  {\LL Above and throughout we use $C$ to denote a generic constant
    independent of the mesh element sizes and the stabilization
    parameter $\tau$. The notations appearing above are defined as
    follows, letting $F_{\max}$ denote the face of $K$ where
    $\tau|_{\d K}$ is maximum:}
  \[
  {\LL \tau_K^{\max}=\max \tau|_{\partial K},
  \quad
  \tau_K^{*}=\max \tau|_{\d K\setminus F_{\max}},
  \quad
  h_K = \diam(K),
  \quad
  h =
  \max_{K \in Th} h_K.}
  \]
  The following error estimate is
  known~\cite{CockbDuboiGopal10,CockbGopalSayas10}.

\begin{theorem}[see~\cite{CockbDuboiGopal10}]
  \label{thm:srccgce}
  Let the exact solution $(\vqf,u^f)$ of~\eqref{eq:theproblem} be in
  $\Hdiv\Omega \cap H^s(\Omega)^n \times H^s(\Omega)$ for some
  $s>1/2$.  Then,
  \begin{align}
%     \label{eq:ee-l}
%     \| \Pm u  - \eta_h^f \|_a & \le\phantom{2} \| \vq - \pv \vq\,\|_c,\\
    \label{eq:ee-u}
    \| u -  u_h^f \|_\oh & \le  C \| u - \pw u \|_\oh + b_\tau
    C \| \vq - \pv \vq \,\|_c,
    \\    \label{eq:ee-q}
    \| \vq - \vq_h^f \|_c & \le 2 \| \vq - \pv \vq \,\|_c,
  \end{align}
  where $\displaystyle{b_\tau = \max\{ 1 + h_K {\tau_K^*} +
      h_K/\taum: K\in\Th\}}$,  and
    {\LL $\|\vq\,\|_c=(c\,\vq,\vq\,)^{1/2}_{\doh}$ with $c=\alpha^{-1}$.}
\end{theorem}

\subsection{The eigenproblem}           \label{ssec:eigenproblem}

The HDG discretization of the eigenproblem~\eqref{eq:lambda} defines
an approximation to the eigenfunction $u_h\in W_h$,
an approximation to the eigenflux $\vq_h\in V_h$, and an
approximation to the eigenfunction trace
$\eta_h\in M_h$, as a nontrivial set of functions
satisfying
\begin{subequations}
\label{eq:hdgeig}
\begin{alignat}{2}
\label{eq:hdgeig-a}
 ( c\, \vq_h,\vr)_{\oh} -
 ( u_h,\dive \vr)_{\oh} +   \ip{\eta_h,\vr\cdot\vn}_{\doh}
 &  = 0,
 &&\Forall \vr\in V_h,
 \\
\label{eq:hdgeig-b}
 -( \vq_h, \grad w)_{\oh} +
 \ip{ \qhat_h\cdot\vn, w}_{\doh}
& =
 \lambda_h\, ( u_h, w)_{\oh}
 &&
 \Forall w \in W_h,\\
\label{eq:hdgeig-c}
 \ip{\mu, \qhat_h \cdot \vn }_{\doh}
   & = 0
   && \Forall \mu \in M_h,
\end{alignat}
\end{subequations}
where $\qhat_h$ is defined by $\qhat_h = \vq_h + \tau \big( u_h -
\eta_h \big) \vn$, cf.~\eqref{eq:qhatf}.  Here, $\lambda_h\in \RRR$ is
the corresponding discrete eigenvalue.

The unique solvability of the source problem \eqref{method} implies that
$\lambda_h$ is nonzero.
One can easily verify that
 $\mu_h$ is an eigenvalue of $T_h$ if and
only if $\mu_h=1/\lambda_h$ for some $\lambda_h$ solving \eqref{eq:hdgeig}.
Moreover, $\lambda_h$ is positive as can be concluded from
the next lemma.
\begin{lemma}
  \label{lem:symm:Th}
  $T_h$ is self-adjoint and positive definite in $L^2(\Omega)$.
\end{lemma}
\begin{proof}
  To show that $(f, T_h g) = (T_hf, g)$, where $(\cdot,\cdot)$ denotes
  the $L^2(\Omega)$-inner product, we calculate as follows:
  \begin{align*}
    (&f, T_h g)
     =  -( \vqf_h, \grad u_h^g)_{\oh} +
    \ip{ \qhat_h^f\cdot\vn, u_h^g}_{\doh},
    && \text{by~\eqref{eq:method-b} with } w=T_hg = u_h^g
    \\
    & =
    (\dive \vqf_h, u_h^g)_\oh
    +\ip{  ( \qhat_h^f - \vqf_h)\cdot\vn, u_h^g}_{\doh}
    &&\text{by integration by parts}
    \\
    & =(c\,\vqg_h,\vqf_h)_\oh
     + \ip{ \eta_h^g,\vqf_h\cdot\vn}_{\doh}
     +\ip{  ( \qhat_h^f - \vqf_h)\cdot\vn, u_h^g}_{\doh}
     &&\text{by~\eqref{eq:method-a} with } \vr=\vqf_h
     \\
     &  =(c\,\vqg_h,\vqf_h)_\oh
     - \ip{\vqf_h\cdot\vn,u_h^g - \eta_h^g}_{\doh}
     + \ip{ \qhat_h^f\cdot\vn, u_h^g - \eta_h^g}_{\doh}
     &&\text{by~\eqref{eq:method-c} with }\mu=\eta_h^g
     \\
     & =
       (c\,\vqg_h,\vqf_h)_\oh
       + \ip{\tau (u_h^f - \eta_h^f), (u_h^g-\eta_h^g)}_{\doh}
     &&\text{by~\eqref{eq:qhatf}}.
 \end{align*}
 The last expression is symmetric in $f$ and $g$ and is non-negative
 if $f=g$. This proves that $T_h$ is self-adjoint and positive
 semidefinite in $L^2(\Omega)$. As already noted previously, zero is
 not an eigenvalue of $T_h$, hence $T_h$ is positive definite.
\end{proof}

\section{Convergence of the spectrum}   \label{sec:convergence}

The convergence of the discrete eigenvalues to the exact ones is
proved by establishing convergence of $T_h$ to~$T$ in operator norm.
(Recall that $T$ and $T_h$ are defined in~\eqref{eq:T}
and~\eqref{eq:Th}, respectively.)  Such operator convergence was used
as the basis for the early analyses of spectral approximations using
conforming methods~\cite{BrambOsbor73,MerciOsborRappa81,Osbor75}. It
has also been used to analyze approximations of eigenvalue problems
using older discontinuous Galerkin methods (like the interior penalty
method)~\cite{AntonBuffaPerug06}. To apply this technique to the HDG
eigenproblem, we need the following basic result.

\begin{theorem}[Operator convergence]
  \label{thm:opcgce}
  Suppose there is an $s>1/2$ such  that
  any solution $(\vqf,
  u^f)$ of~\eqref{eq:theproblem} satisfies
  \begin{equation}
    \label{eq:1}
      \| \vqf \|_{H^s(\Om)} + \| u^f \|_{H^s(\Om)} \le C \| f \|_{L^2(\Om)},
 \end{equation}
 for all $f \in L^2(\Omega).$ Then
  \begin{equation}
    \label{eq:2}
    \| T - T_h \| \le  c_\tau  h^{\min(s,k+1)}
  \end{equation}
  where $\| \cdot \|$ denotes the $L^2(\Om)$-operator norm
  and $c_\tau =C\max\{ 1 + h_K^2 (\tau_K^*)^4 + (\taum)^{-2}: K\in\Th\}^{1/2}$.
\end{theorem}
\begin{proof}
  The convergence results for the HDG source problem imply
  \begin{align*}
    \| Tf  - T_h f \|_{L^2(\Om)}
    &
    \le  C \| u - \pw u \|_{\LL\Th} + b_\tau
    C \| \vq - \pv \vq \,\|_c,
    && \text{by  Theorem~\ref{thm:srccgce}, eq.~\eqref{eq:ee-u},}
    \\
    &
    \le  C  c_\tau h^{\min(s,k+1)}
    ( |q|_{H^s(\Omega)} + | u|_{H^s(\Omega)}),
    && \text{by~\eqref{eq:PIapprx}}.
  \end{align*}
  Hence the result follows from~\eqref{eq:1}.
\end{proof}

Note that assumption~\eqref{eq:1} is a regularity assumption that holds, for
example, when $s \equiv 1$ and $\Omega$ is a polygonal domain
(with no cracks) in $\RRR^2$.

By virtue of Theorem~\ref{thm:opcgce} and the well-known consequences
of operator convergence~\cite{Ansel71,DunfoSchwa88,Osbor75}, we
conclude that the spectrum of $T_h$ approximates that of $T$, i.e.,
there is {\em no ``pollution'' of the spectrum} when it is
approximated by the HDG method.  To formulate the statement of this
approximation precisely in a form we can use later, let us recall some
standard terminology. The ``gap'' between two subspaces $X$ and $Y$ of
$L^2(\Omega)$ is defined by
\[
 \delta(X,Y)
  = \sup_{x \in X} \; \frac{\dist(x,Y)}{\|  x \|_{L^2(\Omega)}}
  = \sup_{y\in Y} \; \frac{\dist(y,X)}{\|  y \|_{L^2(\Omega)}}.
\]
%Here and throughout $\| \cdot\|$ denotes the $L^2(\Omega)$-norm.

Now, suppose $\mu$ is an eigenvalue of $T$ of multiplicity $m$ and let
$\vG$ be a positively oriented circle, contained in the resolvent set
of $T$, centered at $\mu$, and enclosing no other eigenvalue of $T$.
Define two operators, $E_h^\vG$ and $E^\vG$, both on $L^2(\Omega)$, by
the following integrals over $\vG$ in the complex
plane:
\[
\begin{aligned}
E_h^\vG
& = \frac 1 {2 \pi \ii} \oint_\vG ( z - T_h)^{-1} dz,
&\qquad
E^\vG
& = \frac 1 {2 \pi \ii} \oint_\vG ( z - T)^{-1} dz.
\end{aligned}
\]
Hereon, we omit the superscript $\vG$ in $E_h^\vG$ and $E^\vG$ as
$\vG$ will always be taken as stated above.  A well-known result is
that $E$ is a projection onto the eigenspace of $T$ corresponding to
the eigenvalue~$\mu$. The discrete analogue of this result for $E_h$
appears in results collected below.  The collection summarizes a few
consequences of operator convergence.  The arguments proving these
consequences are
standard~\cite{AntonBuffaPerug06,MerciOsborRappa81,Osbor75}, and since
they apply to
%in
the HDG context with few modifications, we shall not
repeat them. We use $R(A)$ to denote the range of any operator $A$ and
$C_\tau$ to denote generic constants independent of $h$, but dependent
on $\tau$.

\begin{corollary}[Convergence of eigenvalues and eigenfunctions]
  \label{cor:eigcgce}
  Let $\mu$ and $\vG$ be as above and let the assumption of
  Theorem~\ref{thm:opcgce} hold for some $s>1/2$.  Then, there exists
  $h_\vG>0$ such that for all $h < h_\vG$ the following statements
  hold:
  \begin{enumerate}

  \item There are exactly $m$ eigenvalues of $T_h$ within $\vG$, which
    we count according to multiplicity and enumerate as $\mu_{h,1},
    \mu_{h,2}, \ldots, \mu_{h,m}$.

  \item \label{item:2} The operator $E_h$ is a projection onto the span of the
    eigenfunctions of $T_h$ corresponding to all the eigenvalues
    $\mu_{h,j}$ for $j=1,\ldots,m$.

  \item \label{item:3} The operator $E_h$ converges to $E$ as $h\rightarrow 0$ and
    \begin{equation}
      \label{eq:12}
          \| E - E_h \| \le
          % C \big\| (T - T_h )\raisebox{-1pt}{$|$}_{\ran(E)}\big\|
          C \| T - T_h \|
          \le C_\tau h^{\min(s,k+1)}.
   \end{equation}

 \item \label{item:4} The exact and discrete eigenspaces (of $\mu$ and
   $\{\mu_{h,j}\}_{j=1}^m $, respectively) are $R(E)$ and $R(E_h)$,
   respectively. The discrete eigenspaces converge in the sense that
    \begin{equation}
      \label{eq:13}
          \delta( R(E), R(E_h) )     \le C_\tau h^{\min(s,k+1)}.
    \end{equation}

    \item If in addition, the eigenfunctions of $\mu$
      have a higher
      regularity index, i.e., if
      \[
      \| \vqf \|_{H^{s_\mu}(\Om)} +
      \| u^f \|_{H^{s_\mu}(\Om)} \le C \| f \|_{L^2(\Om)},
      \qquad\quad\forall
      f \in R(E),
      \]
      with $s_\mu \ge s$,
      then~\eqref{eq:12} can be refined to
      \begin{equation}
        \label{eq:14}
        \big\| (E - E_h )\raisebox{-1pt}{$|$}_{R(E)}\big\|
        \le
        C \big\| (T - T_h )\raisebox{-1pt}{$|$}_{R(E)}\big\|
        \le C_\tau h^{\min(s_\mu,k+1)}.
      \end{equation}
      and consequently~\eqref{eq:13} can be revised to
      $
      \delta( R(E), R(E_h) )
      \le C_\tau h^{\min(s_\mu,k+1)}$.
    \end{enumerate}
\end{corollary}

\section{Rate of convergence of eigenvalues}   \label{sec:rate}

In this section, we prove that under favorable regularity conditions,
the HDG eigenvalues converge at the rate $O(h^{2k+1})$ when we use
polynomials of degree at most~$k\ge 0$ for all variables. To do so, we
begin with the setting of Corollary~\ref{cor:eigcgce} and refine a few
estimates through a duality argument. Accordingly, we keep the same
notations as in Corollary~\ref{cor:eigcgce}, and tacitly assume
throughout this section that the assumptions in the corollary hold.
In particular, recall that $R(E)$ is the eigenspace of $T$
corresponding to an eigenvalue~$\mu$ and $\mu_{h,j}$ are the discrete
eigenvalues near $\mu$.

\begin{theorem}
  \label{thm:ewrate}
  Suppose there is an  $s_\mu>1/2$ such that
  \begin{equation}
    \label{eq:3}
      \| \vqf \|_{H^{s_\mu}(\Om)} + \| u^f \|_{H^{s_\mu+1}(\Om)} \le C \| f \|_{L^2(\Om)},
      \qquad\quad\forall
      f \in R(E).
  \end{equation}
  Then there is an $h_\mu>0$ such that for all $h<h_\mu$,
  \begin{equation}
    \label{eq:4}
    |\mu - \mu_{h,j} | \le C h^{\min(s_\mu,k+1) + \min(s_\mu,k)}.
  \end{equation}
\end{theorem}

This is the main result of this section, and in the remainder of this
section, we prove it. As we shall see, we are able to apply the needed
duality techniques for this proof, thanks to properties of the
projection operator $\pr_h(\vq,u)$ of~\cite{CockbGopalSayas10},
recalled in Subsection~\ref{ssec:source}.  Given any eigenfunction $e
\in R(E)$, let us denote its corresponding flux by
\begin{equation}
  \label{eq:7}
  \vq_e = -\alpha \grad e.
\end{equation}
Then $\pw e$ is the component of $\pr_h(\vq_e,e)$ in $W_h$.
Define
\[
J_h e = E_h \pw e
\]
for all $e$ in $R(E)$.  Using these notations, we begin the
analysis with the following two lemmas.

\begin{lemma}
  \label{lem:symm}
  $E_h$ is a self-adjoint operator in $L^2(\Omega)$.
\end{lemma}
\begin{proof}
  By Lemma \ref{lem:symm:Th}, $T_h$ is self-adjoint.  It is well-known
  that the spectral projection of any normal operator is
  self-adjoint~\cite{Kato95}. Since $E_h$ is the spectral projection
  of $T_h$, the lemma follows.
\end{proof}

{\LL From now on, to simplify notation, let us abbreviate the norm $\|
  \cdot \|_{L^2(\Omega)}$ to simply $\| \cdot \|$, whenever it cannot
  be confused with the previously defined $L^2(\Omega)$-operator norm, which we
  continue to also denote by $\| \cdot\|$.}

\begin{lemma}
  \label{lem:Jh}
  Suppose~\eqref{eq:3} holds. Then, there exists an $h_0>0$ such that
  for all $h < h_0$, the operators
  \mbox{$J_h : R(E) \longmapsto R(E_h)$}
  and
  \mbox{$E_h \raisebox{-1pt}{$|$}_{R(E)} : R(E) \longmapsto R(E_h)$}
  are bijections, and there are $h$-independent constants $C_{\tau,j}$
  such that
  \begin{align}
    \label{eq:8}
    C_{\tau,1} \| e \| & \le \| J_h e \| \le C_{\tau,2} \| e \|,
    \\
    \label{eq:18}
    C_{\tau,3} \| e \| & \le \| E_h e \| \le \| e \|,
  \end{align}
  for all $e$ in $R(E)$.
\end{lemma}
\begin{proof}
  By item~\ref{item:4} of Corollary~\ref{cor:eigcgce}, the gap between
  $R(E)$ and $R(E_h)$ becomes less than one when $h$ is small enough,
  and consequently $\dim(R(E)) = \dim(R(E_h))$ (see,
  e.g.,~\cite[Lemma~221]{Kato58}).  Therefore, to prove the stated
  bijectivity, we only need to prove injectivity.

  In preparation, we recall that by Lemma~\ref{lem:symm}, $E_h$ is
  self-adjoint, and by Corollary~\ref{cor:eigcgce}(\ref{item:2}),
  $E_h$ is a projector. Hence $E_h$ is an orthogonal
  projector. Orthogonal projectors have unit norm, hence
 \begin{equation}
    \label{eq:5}
        \| E_h v \| \le \| v\|,\qquad \forall v \in L^2(\Omega).
  \end{equation}
  A second preparatory inequality we need is
  \begin{equation}
    \label{eq:19}
    \| e - \pw e \| \le C_\tau h^r \| e \|,\qquad \forall e \in R(E),
  \end{equation}
  with $r=\min(s_\mu,k+1).$ This is a consequence
  of~\eqref{eq:PIapprx-b}, by which
  \begin{align*}
    \| e - \pw e \|
    & \le C_\tau h^r
    ( | e |_{H^{s_\mu(\Omega)}} +| \vq_e |_{H^{s_\mu}(\Omega)} ),
  \end{align*}
  where $\vq_e$ is as in~\eqref{eq:7}. Thus,~\eqref{eq:19} follows
  from the regularity assumption~\eqref{eq:3}.

  Let us now prove~\eqref{eq:8},  beginning with the lower bound.
  \begin{align*}
    \| J_h e \|
    & \ge  \| e \| - \| e - E_h \pw e \|
    \\
    & = \| e \| - \| (E e - E_h e) + E_h (e -\pw e ) \|
    && \text{since $e \in R(E)$}
    \\
    & \ge  \| e \| - \| (E - E_h) e \| - \|e -\pw e\|
    &&\text{by~\eqref{eq:5}}
    \\
    & \ge
    (1 - C_\tau h^r )\| e \|
    &&\text{by~\eqref{eq:19} and~\eqref{eq:14}.}
  \end{align*}
  Therefore, the lower bound follows by choosing small enough~$h$.
  The injectivity of $J_h$ is an obvious consequence of this lower
  bound.  The upper bound of~\eqref{eq:8} immediately follows by
  combining~\eqref{eq:19},
  \[
  \| \pw e \| \le (1+C_\tau h^r ) \| e\|\qquad \forall e \in R(E),
  \]
  with~\eqref{eq:5}.

  To finish the proof, note that the upper bound in~\eqref{eq:18} is
  already proved in~\eqref{eq:5}. The lower bound in~\eqref{eq:18}
  (and the consequent injectivity of $E_h$ on $R(E)$) follows by an
  simpler argument similar to the above.
\end{proof}

\bigskip

\begin{proof}[Proof of Theorem~\ref{thm:ewrate}]

  As a first step, we define two {\em finite} dimensional operators whose
  eigenvalues are $\mu$ and $\mu_{h,j}$. Let $h$ be so small that we can conclude
  by Lemma~\ref{lem:Jh} that
  $J_h^{-1}: R(E_h)\to R(E)$ exists.
  Set $\hat T = T|_{R(E)}$ and
  $\hat T_h = J_h^{-1} T_h J_h|_{R(E)}$.  Both the operators
  \[
  \hat T: R(E)  \to  R(E)
  \qquad\text{and}
  \qquad
  \hat T_h : R(E)  \to  R(E)
  \]
  are finite dimensional.  $\hat T$ has $\mu$ as its (only) eigenvalue
  of of multiplicity $m$. Moreover, it is easy to see that
  $\mu_{h,j}$'s are the eigenvalues of $\hat T_h$. Hence, by the
  Bauer-Fike theorem~\cite{BauerFike60},
  \begin{equation}
    \label{eq:9}
    | \mu - \mu_{h,j}| \le C \| \hat T - \hat T_h \|.
  \end{equation}
  The remainder of the proof bounds the above right hand side
  appropriately.

  To this end,  let $f\in R(E)$, and consider $(\hat T - \hat T_h )f$. Then,
  \begin{align}
    \nonumber
    C_{\tau,1}\|( \hat T - \hat T_h )f \|
    & \le  \| J_h ( \hat T - \hat T_h )f \|
    && \text{by~\eqref{eq:8} of Lemma~\ref{lem:Jh}}
    \\     \nonumber
    & = \| E_h \pw T f - T_h E_h\pw  f \|
    \\ \label{eq:10}
    & = \| E_h (\pw T f - T_h   \pw f) \|
    &&\text{as $T_h$ and $E_h$ commute}.
  \end{align}
  We bound the norm in~\eqref{eq:10} by duality, as follows.  By
  Lemma~\ref{lem:Jh}, any $g_h$ in $R(E_h)$ can be written as $E_h g$
  for some $g$ in $R(E)$.  Therefore,
  \begin{align}
    \nonumber
    \| E_h (\pw T f - T_h   \pw f) \|
    & = \sup_{g_h\in R(E_h)}
        \frac{ ( E_h (\pw T f - T_h   \pw f), g_h ) }{ \| g_h \| }
        \\
        \nonumber
    & = \sup_{g\in R(E)}
        \frac{ ( E_h (\pw T f - T_h   \pw f), E_h g ) }{ \| E_h g \| }
        \\
        \label{eq:11}
    & \le \frac{1}{C_{\tau,3}}\;
       \sup_{g\in R(E)}
        \frac{ ( \pw T f - T_h   \pw f, E_h g ) }{ \| g \| }.
  \end{align}
  Note that to obtain the numerator above, we used the
  self-adjointness of $E_h$ given by Lemma~\ref{lem:symm}, while to
  obtain the denominator, we used Lemma~\ref{lem:Jh}.

  It will now be convenient to split the numerator in~\eqref{eq:11}
  into several terms.  With $f$ and $g$ in $R(E)$, we write
  \begin{align*}
    (\pw T f - T_h   \pw f, E_h g )
    & = (\pw T f - T_h f, E_h g ) + (T_h (f-  \pw f), E_h g )
    \\
    & = t_1 + t_2 + t_3 + t_4
  \end{align*}
  where
  \begin{align*}
    t_1 & = (\pw T f - T_h f, E_h g -g ) \\
    t_2 & = (\pw T f - T_h f, g ) \\
    t_3 & = (T_h (f-  \pw f), E_h g -g ) \\
    t_4 & = (T_h (f-  \pw f), g ),
 \end{align*}
 and proceed to estimate the $t_i$'s, beginning with $t_1$.
 Let $r_1=\min(s_\mu,k+1)$. We  use Theorem~\ref{thm:srccgce}
 and~\eqref{eq:14} to get
 \begin{align*}
   t_1 & = ( \pw u^f - u_h^f, (E_h-E) g )
    \\
    & \le C_\tau h^{r_1} ( |u^f|_{H^{r_1}(\Omega)} + |q^f|_{H^{r_1}(\Omega)})  h^{r_1} \| g \|
    \\
    & \le C_\tau h^{2{r_1}} \| f \|  \;\| g \|,
 \end{align*}
 by~\eqref{eq:3}.

 Next, for $t_2$, we use a previously known duality
 identity~\cite[Lemma~4.1]{CockbGopalSayas10} by which
 \begin{align*}
   t_2=(\pw u^f - u_h^f, g )
    & = (c\, (\vqf - \vqf_h), \pv \vqg - \vqg) +
       ( \vqf - \pv \vqf, \grad u^g -\grad w_h )_\oh
  \end{align*}
  for any $w_h \in W_h.$ The first term on the right hand side can be
  bounded by $C_\tau h^{2r_1} \| f \| \;\| g \|$
  using~\eqref{eq:PIapprx} and Theorem~\ref{thm:srccgce}. For the
  second term, we use the standard result
  \[
  \inf_{w_h \in W_h}  \| \grad u^g -\grad w_h \| \le C  h^{r_0} |u^g|_{H^{r_0+1}(\Omega)}
  \]
  with $r_0 = \min(s_\mu,k)$, to get
  \begin{align*}
    t_2
    & \le
    C_\tau h^{2r_1} \| f\| \, \|g \|
     + C_\tau h^{r_1} ( |u^f|_{H^{r_1}(\Omega)} + |q^f|_{H^{r_1}(\Omega)})
        \;h^{r_0} |u^g|_{H^{r_0+1}(\Omega)}
        \\
    & \le     C_\tau (h^{2r_1} + h^{r_1+r_0})  \| f\| \, \|g \|
  \end{align*}
  by the regularity assumption~\eqref{eq:3}.

  Now consider $t_3$. Since $T$ is a bounded operator in
  $L^2(\Omega)$, by Theorem~\ref{thm:opcgce}, we have
  \begin{equation}
    \label{eq:15}
    \| T_h v \| \le C_\tau \| v \| \qquad\forall v \in L^2(\Omega).
  \end{equation}
  Hence,
  \begin{align*}
    t_3
    &  = (T_h(f - \pw f ), (E_h-E)g)
    \le  C_\tau \| f - \pw f \| \; C h^{r_1} \| g \|
  \end{align*}
  by~\eqref{eq:14} and~\eqref{eq:3}. To control $ f - \pw f $ let us
  first note that $u^f = \mu f$ since $f \in R(E)$. Hence absorbing
  the $\mu$-dependence into $C_\tau$, we can write
  \begin{equation}
    \label{eq:17}
      \| f - \pw f \| \le C_\tau h^{r_1}
      ( |u^f|_{H^{r_1}(\Omega)} + |q^f|_{H^{r_1}(\Omega)}).
  \end{equation}
 Using the regularity assumption~\eqref{eq:3} again, we conclude that
  \[
  t_3 \le C h^{2 r_1} \| f \|\, \|g\|.
  \]

  Finally, consider $t_4$. By   Lemma~\ref{lem:symm:Th},
  $T_h$ is self-adjoint, so
  \begin{align}
    \nonumber
    t_4
    & = (f - \pw f , u_h^g)
    \\ \label{eq:16}
    & = (f - \pw f , u_h^g - u_{h,k-1}^g )
  \end{align}
  where $u_{k-1}^g \in L^2(\Omega)$ is the solution of the HDG method
  using polynomials of degree $k-1$ in place of $k$ for all $k\ge 1$,
  while when $k=0$, we set $u_{k-1}^g=0$.  We may introduce
  $u_{k-1}^g$ in~\eqref{eq:16} because of~\eqref{eq:proj2}. By
  Theorem~\ref{thm:srccgce} and~\eqref{eq:3},
  \[
  \| u_h^g - u_{h,k-1}^g \| \le
   \| u_h^g - u \| + \| u - u_{h,k-1}^g \|
   \le C_\tau h^{r_0} \| g \|.
  \]
  Hence, using also~\eqref{eq:17}, we find that
  \[
  t_4 \le C_\tau h^{r_0+r_1} \| f \|\, \| g \|.
  \]
  Combining the estimates for all $t_i$ to bound the right hand
  side of~\eqref{eq:11}, we find from~\eqref{eq:10} that
  \[
  \| (\hat T - \hat T_h)f \| \le C_\tau h^{r_0+r_1} \| f \|.
  \]
  Using this in~\eqref{eq:9} the proof is finished.
\end{proof}

\section{The hybridized eigenproblem}  \label{sec:hybrid}

The main advantage the HDG method possesses over other DG methods, in
source problems, is that one can eliminate (all) the interior
variables ($u_h^f,\vqf_h$) to obtain a single equation for the
Lagrange multiplier $\eta_h^f$. Since $\eta_h^f$ is defined on mesh
faces of dimension~$n-1$, the reduced system for $\eta_h^f$ can be
significantly smaller in size for high degrees~$k$. It is natural to
ask if this reduction in system size can be carried over to the
eigenproblem. In this section we study this issue. The analysis here
is modeled after~\cite{hrt:2010}.

First, let us review the above mentioned elimination result for the
source problem. Define {\em local solution operators} $\QQ: M_h
 \to  V_h$, $\UU: M_h  \to  W_h$, $\Qw: L^2(\Omega)  \to  V_h$,
$\Uw: L^2(\Omega)  \to  W_h$, element by element, by the following
two systems: For any $K$ in $\oh$,
 \begin{subequations}
 \label{eq:lift}
   \begin{align}
  \label{eq:lift-a}
     ( c\, \QQ\mu,\vr)_{K}
      -
     ( \UU\mu,\dive \vr)_{K}
      & = - \ip{\mu,\vr\cdot\vn}_{\d K}
      &&\text{ for all } \vr \in V_h, \\
 \label{eq:lift-b}
  (w,\dive \QQ\mu)_{K}
     +  \ip{ \tau ( \UU\mu - \mu),  w}_{\d K}
     & =  0
     &&\text{ for all } w \in W_h,
   \end{align}
 \end{subequations}
 and
 \begin{subequations}
 \label{eq:lift:f}
   \begin{align}
  \label{eq:lift-a:f}
     ( c\, \Qw f,\vr)_{K}
      -
     ( \Uw f,\dive \vr)_{K}
      & = 0
      &&\text{ for all } \vr \in V_h, \\
 \label{eq:lift-b:f}
  (w,\dive \Qw f)_{K}
     +  \ip{ \tau \Uw f,  w}_{\d K}
     & =  (f, w)_K
     &&\text{ for all } w \in W_h.
   \end{align}
 \end{subequations}
 The properties of these operators are amply discussed
 in~\cite{CockbGopalLazar09,CockbDuboiGopal10} and will not be
 repeated. Let $a_h(\eta, \mu) = (c\, \QQ\eta, \QQ\mu)_\oh
      + \ip{ \tau (\UU\eta- \eta), (\UU\mu - \mu)}_\doh$
      and $b_h(\mu) = (f, \UU\mu)_\oh$.

 \begin{theorem} [The reduced source problem
   -- \mbox{see~\cite[Theorem~2.1]{CockbGopalLazar09}}]
          \label{thm:reduction:S}
          The functions $\vqf_h \in V_h, \; u_h^f \in W_h,$
  and $ \eta_h^f \in M_h$
  satisfy~\eqref{method} if and only if $\eta_h^f$ is
  the unique function in $M_h$ satisfying
 \begin{gather}
   \label{eq:lam}
   a_h(\eta_h^f, \mu) = b_h(\mu) \Forall \mu \in M_h,
   \\
   \vqf_h  = \QQ\eta_h^f + \Qw f
   \quad\text{ and } \quad
   \label{eq:u-recovery}
   u_h^f  = \UU\eta_h^f + \Uw f.
 \end{gather}
 \end{theorem}

 Our aim is to obtain an analogue of Theorem~\ref{thm:reduction:S} for
 the eigenproblem. To this end, we will need the next lemma.  Let
 \[
 d_{h\tau}^K=1+(\tau_K^\textrm{max}h_K)^{-1/2},\quad
 c_{h\tau}^K=1+(\tau_K^*h_K)^{1/2},
 \quad
 \|\mu\|_{h,K}^2=\|\mu\|_{L^2(\d K)}^2\frac{|K|}{|\d K|}.
 \]
 We assume throughout this section there is a mesh-independent $C_*>0$
 such that
 \[
 (d_{h\tau}^K)^2 h_K \le C_*
 \quad \text{ and }\quad
 c_{h\tau}^K \le C_*.
 \]
 This certainly holds for the most commonly used choice of $\tau
 \equiv 1$.

 \begin{lemma}
   \label{lem:bounds}
   The local solution operators satisfy the following bounds: There
   are positive constants $C_0$ and $C_1$ (independent of $h_K$) such that
    \begin{align}
      \label{eq:Uw:bound}
      \|\Uw f\|_{\LL{L^2(K)}}
      &\leq C_0 (d_{h\tau}^K)^2 h^2_K \|f\|_{\LL{L^2(K)}}
      && \forall \; f\in L^2(K)\\
      \label{eq:U:bound}
      \| \UU \mu \|_{\LL{L^2(K)}}
      & \le C_1 c_{h\tau}^K \| \mu\|_{h,K} && \forall \; \mu\in M_h.
\intertext{Moreover, whenever $0< \kappa< 1/(C_0 (d_{h\tau}^K)^2 h_K^2)$, the operator $I - \kappa \Uw$
is invertible and}
    \label{eq:inversebound}
      \| (I - \kappa \Uw)^{-1} f \|_{\LL{L^2(K)}}
      & \le \frac{1}{1- \kappa C_0 (d_{h\tau}^K)^2 h_K^2} \| f \|_{\LL{L^2(K)}}
      && \forall \; f\in L^2(K).
    \end{align}

\end{lemma}
\begin{proof}
  The estimates~\eqref{eq:Uw:bound} and~\eqref{eq:U:bound} are proved
  in~\cite{CockbDuboiGopal10}. To prove~\eqref{eq:inversebound},
  we use~\eqref{eq:Uw:bound}, by which
  \begin{align*}
    \| \kappa \Uw f \|_{\LL{L^2(K)}}
    \le \gamma \| f \|_{\LL{L^2(K)}}.
  \end{align*}
  where $\gamma = \kappa (d_{h\tau}^K)^2 C_0 h_K^2$. By the given
  assumption $\gamma<1$, so the $L^2(\Omega)$-operator norm of $\kappa
  \Uw$ is less than one. Consequently, $I - \kappa \Uw$ is invertible
  and the norm of the inverse is less than $(1-\gamma)^{-1}$.
\end{proof}

\begin{theorem}
  \label{thm:reduction:eig}
  Let $0< C_3 < (C_0 C_*)^{-1}$ and let $\lambda_h$ be any positive
  number less than $C_3/h$. Then
  $I-\lambda_h \Uw$ is invertible, and moreover,
  $\lambda_h$
  satisfies
  \begin{equation}
    \label{eq:m-a-form}
    a_h(\eta_h,\mu) = \lambda_h \,
    ( (I-\lambda_h \Uw)^{-1} \UU\eta_h, \UU\mu)_\oh
  \qquad\forall \mu \in M_h
  \end{equation}
  with some nontrivial $\eta_h$ in $M_h$, if and only if
  the number
  $\lambda_h$ and the functions
  \begin{equation}
    \label{eq:funs}
  \eta_h,\quad
  u_h  = (I - \lambda_h \Uw )^{-1} \UU\eta_h, \quad\text{ and }\quad
  q_h  = \QQ\eta_h + \lambda_h \Qw u_h
  \end{equation}
  together solve the HDG eigenproblem~\eqref{eq:hdgeig}.
\end{theorem}
\begin{proof}
  The argument proceeds as in~\cite{hrt:2010}, so we will be brief.
  Setting $f = \lambda_h u_h$ in~\eqref{method} and applying
  Theorem~\ref{thm:reduction:S}, \eqref{eq:u-recovery}, we get $u_h =
  \UU \eta_h + \Uw (\lambda_h u_h)$. Hence,
  \begin{equation}
    \label{eq:20}
      u_h = (I - \lambda_h \Uw)^{-1} \UU \eta_h,
  \end{equation}
  where the inverse exists by Lemma~\ref{lem:bounds}, whenever
  $\lambda_h< 1/(C_0 (d_{h\tau}^K)^2 h_K^2)$.  Note that the later
  inequality is satisfied whenever $\lambda_h < C_3/h$ because $
  \lambda_h C_0 d_{h\tau}^K h_K^2 \le \lambda_h C_0 C_* h_K \le
  \lambda_h C_0 C_* h < C_3 C_0 C_* <1.  $ Now, using~\eqref{eq:20} in
  the right hand side of~\eqref{eq:lam} and proceeding as
  in~\cite{hrt:2010}, the proof is finished.
\end{proof}

Theorem \ref{thm:reduction:eig} is the analogue of
Theorem~\ref{thm:reduction:S} for the eigenproblem.  It shows, roughly
speaking, that the reduced form of the eigenproblem,
namely~\eqref{eq:m-a-form}, does not lose eigenvalues up to
$O(1/h)$. In particular, the physically important lower range of the
spectrum is preserved.  The difficulty with~\eqref{eq:m-a-form} is
that in spite of being a {\em smaller} system~\eqref{eq:hdgeig}, it is
a {\em nonlinear} eigenvalue problem. Fortunately, the result we
present next shows that good initial guesses for the nonlinear
eigenproblem~\eqref{eq:m-a-form} can be calculated by solving a
standard symmetric generalized eigenproblem. This standard eigenproblem
is to find an eigenvalue $\lt_h>0$ and a corresponding eigenfunction
$\et_h\not\equiv 0$ in $M_h$ satisfying
\begin{equation}
  \label{eq:tempt}
  a_h(\et_h, \mu)
   = \lt_h (\UU\et_h,\UU\mu)_\oh \qquad \forall \mu \in M_h.
\end{equation}
and algorithms for solving \eqref{eq:tempt} are well
developed. Although $\lt_h$ may not equal $\lambda_h$, it is a good
approximation to $\lambda_h$ (see Theorem~\ref{thm:init}) and hence
can be used as an initial iterate for iterative algorithms
(such as those studied in~\cite{hrt:2010}) for the nonlinear eigenproblem.

\begin{theorem}
  \label{thm:init}
  Let $0< C_3 < (C_0 C_*)^{-1}$ and consider any eigenvalue
  $\lambda_h$ of~\eqref{eq:hdgeig} satisfying $\lambda_h<C_3/h$.  Then
  there are two constants, $h_0>0$ (depending on $\lambda_h$) and
  $C>0$ (independent of $\lambda_h$), such that for all $h \le h_0$,
  there is an eigenvalue $\lt_h$ of~\eqref{eq:tempt} satisfying
  \begin{equation}
  \label{eq:est:init}
  \frac{ |\lambda_h - \lt_h  |}{\lambda_h }\,
  \le \,\bbchange{C\,\lambda_h  \lt_h \, h}.
  \end{equation}
\end{theorem}

\begin{proof}
  The idea, as in~\cite{hrt:2010}, is to compare two operators which
  have $\lambda_h$ and $\lt_h$ as eigenvalues. The operators are
  $\Shk:M_h  \to  M_h$ and $\tilde{S}_h$: $M_h  \to  M_h$ defined
  by
  \begin{align*}
    a_h(\Shk\mu,\gamma)
    & = ( (I-\kappa \Uw )^{-1} \UU \mu,\UU\gamma)_\oh
    && \forall\gamma\in M_h,
    \\
    a_h( \tilde{S}_h \mu, \gamma)
    & = (\UU\mu,\UU\gamma)_\oh
    && \forall \gamma \in M_h.
  \end{align*}
  We will use $\Shk$ with $\kappa = \lambda_h$, noting that the
  inverse of $I- \lambda_h \Uw$ appearing in the definition of $\Shl$
  exists since $\lambda_h < C_3/h$ (due to Lemma~\ref{lem:bounds} --
  see also proof of Theorem~\ref{thm:reduction:eig}).  Both operators
  are self-adjoint in the $a_h(\cdot,\cdot)$-inner product: The
  self-adjointness of $\tilde S_h$ is obvious. It is also easy to see
  that $(f, \Uw g)_K = ( c\, \Qw g,\Qw f)_{K} + \ip{ \tau \Uw f, \Uw
    g}_{\d K}$ due to~\eqref{eq:lift}, so $\Uw$ in self-adjoint in the
  $L^2(\Omega)$-inner product. Therefore, $\Shl$ is self-adjoint in
  the $a_h(\cdot,\cdot)$-inner product.

  Now, \eqref{eq:m-a-form} and~\eqref{eq:tempt} imply, respectively,  that
  \[
  \lambda_h^{-1} \in \sigma(\Shl)
  \quad\text{ and }\quad
  \lt_h^{-1} \in \sigma( \tilde S_h),
  \]
  hence, by Weyl's theorem~\cite{Weyl12} on eigenvalues of
  self-adjoint operators, we conclude that
\begin{equation}
  \label{eq:weyl}
  \left|
    \lambda_h^{-1} - \lt_h^{-1}\right|
  \le \| \Shl - \tilde{S}_h \|_a
  \equiv \sup_{0\ne\gamma, \mu\in M_h}
  \frac{ a_h( (\Shl - \tilde{S}_h )\gamma,\mu )}{ a_h(\gamma,\gamma)^{1/2}
    a_h(\mu,\mu)^{1/2}}.
\end{equation}
In the remainder of the proof, we bound the right hand side.

By Lemma~\ref{lem:bounds},
choosing $h_0$ appropriately small,
we have, for all $h \le h_0$,
\begin{align*}
  a_h( (\Shl -\tilde{S}_h)\gamma, \mu)
  & =  ( \lambda_h \Uw ( I- \lambda_h \Uw)^{-1} \LU\gamma, \LU \mu )
  \\
  & \le \frac{\lambda_h C_0 C_* h_K}{ 1 - \lambda_h C_0 C_* h_K}
  \|\UU \gamma \| \, \| \UU \mu \|
  &&\text{ by~\eqref{eq:Uw:bound} and~\eqref{eq:inversebound}},
  \\
  &
  \le C \lambda_h h  \|\UU \gamma \|\, \| \UU \mu \|
  \\
  & \le
  C \lambda_h h  \|\gamma \|_h \| \mu \|_h
  &&\text{ by~\eqref{eq:U:bound}.}
\end{align*}
where $\| \mu \|_h^2=\sum_K \| \mu \|_{h,K}^2$.  We now use the
\Poincare-type estimate~\cite[Theorem~3.4]{CockbDuboiGopal10}, $\| \mu
\|_h^2 \le C a_h(\mu,\mu)$ to conclude that
\[
a_h( (\Shl - \tilde{S}_h)\gamma, \mu) \le C \lambda_h\ h\
a_h(\gamma,\gamma)^{1/2} a_h(\mu,\mu)^{1/2}.
\]
Returning to~\eqref{eq:weyl} and using this estimate, the theorem is proved.
\end{proof}

\section{Numerical experiments}   \label{sec:numer}

In this section, we present numerical results to illustrate the
theoretical results of the previous sections. We consider model
eigenproblems on a square and an L-shaped domain and compute the
spectral approximations using the HDG discretization. The model
problems are the same as those considered in~\cite{hrt:2010} so as to
facilitate comparison with the HRT mixed method. We also present a local
postprocessing technique that enhances the eigenfunction and
eigenvalue accuracy beyond the convergence orders predicted by the
theory. We begin by describing this postprocessing in the next
subsection and present the numerical results in the later
subsections.

% The theoretical results show that the approximate gradient $\vec{q}_h$ converges with the same order as the approximate eigenfunction $u_h$. We shall exploit this convergence property of the HDG method to define a simple local postprocessing that allows us to obtain faster convergent approximations of both the eigenfunctions and eigenvalues.

\subsection{Local postprocessing}   \label{ssec:postprocessing}

To postprocess the eigenfunction, we are motivated by the theoretical
results that show that the approximate gradient $\vec{q}_h$ converges
at the same order as the approximate eigenfunction~$u_h$.
Accordingly, following~\cite{Stenb91}, we define (element by element)
the postprocessed eigenfunction
$u_h^\ast \in P_{k+1}(K)$ by
\begin{subequations}
  \label{eq:PP}
\begin{align}
  \label{eq:PP1}
  (\grad u_h^\ast, \grad w)_K  & = -(c\,\vq_h, \grad w)_K, \quad
  \forall w \in P_{k+1}(K),\\
  \label{eq:PP2}
  (u^\ast_h, 1)_K & = ( u_h, 1)_K,
\end{align}
\end{subequations}
for all elements $K\in \Th$.  A convergence theory for this
postprocessing (that predicts that $u_h^\ast$ converges at the rate
$O(h^{k+2})$ for $k \ge 1$) is available for solutions of the source
problem~\cite{Stenb91} and for the HRT mixed
eigenproblem~\cite{hrt:2010}.  Next, we define a postprocessed
eigenflux  $\vec{q}_h^{\, \ast}$ as the unique element of $[P_k(K)]^n +
\vec{x} P_k(K)$ satisfying
\begin{equation}
\label{eq49k}
\begin{split}
\langle (\vec{q}^{\, \ast}_h - \widehat{{q}}_h) \cdot \vec{n},\mu \rangle_{F} & = 0, \quad \forall \mu \in P_k(F), \ \forall F \subseteq \partial K, \\
(\vec{q}^{\, \ast}_h - \vec{q}_h  ,\vec{v}\,)_K  & = 0, \quad \forall \vec{v} \in [P_{k-1}(K)]^d
\quad\mbox{ if } k \ge 1,
\end{split}
\end{equation}
for all elements $K\in \Th$. Note that $\vec{q}_h^{\, \ast}$ is
$H(\rm{div},\Omega)$-conforming.

Using this postprocessed eigenfunction and eigenflux, we are now
motivated by the Rayleigh quotient to define the following expression
for computing an  approximate  eigenvalue:
\begin{equation}
\label{eq:post}
\lambda_h^\ast =  \frac{ (\alpha \grad u_h^\ast, \grad u_h^\ast)_{\oh} +   \ip{ \vec{q}_h^{\, \ast} \cdot \vn, u_h^\ast}_{\doh} }{( u_h^\ast, u_h^\ast)_{\oh}} .
\end{equation}
As the numerical results below indicate, this postprocessed eigenvalue
$\lambda_h^\ast$ can be a superior approximation than
$\lambda_h$.

\subsection{Square domain}   \label{ssec:square}

{\footnotesize{
\begin{table}%[thbp]
  \begin{center}
  \scalebox{0.9}{%
    $\begin{array}{|c|c||c c  | c c|  c c|  c c|}
    \hline
  \mbox{degree} & \mbox{mesh} & \multicolumn{2}{|c}{\mbox{first mode}} &  \multicolumn{2}{|c}{\mbox{second mode}} & \multicolumn{2}{|c|}{\mbox{fourth mode}} & \multicolumn{2}{|c|}{\mbox{sixth mode}} \\
    k & \ell & \mbox{error}  & \mbox{order} & \mbox{error}  & \mbox{order} & \mbox{error}  & \mbox{order} & \mbox{error}  & \mbox{order} \\
 \hline
& 0  &  4.49\mbox{e-}2  &  --  &  1.26\mbox{e-}0  &  --  &  3.03\mbox{e-}0  &  --  &  4.58\mbox{e-}0  &  --  \\
 & 1  &  3.45\mbox{e-}2  &  0.38  &  7.11\mbox{e-}1  &  0.82  &  1.86\mbox{e-}0  &  0.70  &  3.00\mbox{e-}0  &  0.61  \\
0 & 2  &  2.05\mbox{e-}2  &  0.75  &  3.82\mbox{e-}1  &  0.90  &  1.07\mbox{e-}0  &  0.79  &  1.78\mbox{e-}0  &  0.75  \\
 & 3  &  1.11\mbox{e-}2  &  0.89  &  1.99\mbox{e-}1  &  0.94  &  5.84\mbox{e-}1  &  0.88  &  9.81\mbox{e-}1  &  0.86  \\
 & 4  &  5.74\mbox{e-}3  &  0.95  &  1.01\mbox{e-}1  &  0.97  &  3.06\mbox{e-}1  &  0.93  &  5.16\mbox{e-}1  &  0.93  \\
\hline
 & 0  &  5.97\mbox{e-}3  &  --  &  1.40\mbox{e-}1  &  --  &  7.35\mbox{e-}1  &  --  &  1.77\mbox{e-}0  &  --  \\
 & 1  &  8.44\mbox{e-}4  &  2.82  &  1.58\mbox{e-}2  &  3.14  &  9.52\mbox{e-}2  &  2.95  &  2.36\mbox{e-}1  &  2.91  \\
1 & 2  &  1.10\mbox{e-}4  &  2.94  &  1.84\mbox{e-}3  &  3.11  &  1.11\mbox{e-}2  &  3.10  &  2.65\mbox{e-}2  &  3.15  \\
 & 3  &  1.39\mbox{e-}5  &  2.98  &  2.20\mbox{e-}4  &  3.06  &  1.32\mbox{e-}3  &  3.07  &  3.08\mbox{e-}3  &  3.10  \\
 & 4  &  1.75\mbox{e-}6  &  2.99  &  2.69\mbox{e-}5  &  3.03  &  1.61\mbox{e-}4  &  3.04  &  3.71\mbox{e-}4  &  3.06  \\
\hline
 & 0  &  1.38\mbox{e-}4  &  --  &  3.65\mbox{e-}3  &  --  &  4.15\mbox{e-}2  &  --  &  1.27\mbox{e-}1  &  --  \\
 & 1  &  4.53\mbox{e-}6  &  4.93  &  1.04\mbox{e-}4  &  5.13  &  1.26\mbox{e-}3  &  5.04  &  3.51\mbox{e-}3  &  5.18  \\
2 & 2  &  1.43\mbox{e-}7  &  4.98  &  3.12\mbox{e-}6  &  5.06  &  3.82\mbox{e-}5  &  5.05  &  1.02\mbox{e-}4  &  5.10  \\
 & 3  &  4.50\mbox{e-}9  &  4.99  &  9.52\mbox{e-}8  &  5.03  &  1.17\mbox{e-}6  &  5.03  &  3.06\mbox{e-}6  &  5.06  \\
 & 4  &  1.41\mbox{e-}10  &  5.00  &  2.94\mbox{e-}9  &  5.02  &  3.60\mbox{e-}8  &  5.02  &  9.37\mbox{e-}8  &  5.03  \\
\hline
 \end{array} $
}
\end{center}{$\phantom{|}$}
     \caption{Convergence of the approximate eigenvalues $\lambda_h$ for $\tau=1$.}
   \label{tab1}
\end{table}
\begin{table}%[thbp]
  \begin{center}
  \scalebox{0.9}{%
    $\begin{array}{|c|c||c c  | c c|  c c|  c c|}
    \hline
  \mbox{degree} & \mbox{mesh} & \multicolumn{2}{|c}{\mbox{first mode}} &  \multicolumn{2}{|c}{\mbox{second mode}} & \multicolumn{2}{|c|}{\mbox{fourth mode}} & \multicolumn{2}{|c|}{\mbox{sixth mode}} \\
    k & \ell & \mbox{error}  & \mbox{order} & \mbox{error}  & \mbox{order} & \mbox{error}  & \mbox{order} & \mbox{error}  & \mbox{order} \\
 \hline
& 0  &  3.42\mbox{e-}1  &  --  &  1.79\mbox{e-}0  &  --  &  3.56\mbox{e-}0  &  --  &  5.23\mbox{e-}0  &  --  \\
 & 1  &  2.04\mbox{e-}1  &  0.75  &  1.13\mbox{e-}0  &  0.66  &  2.42\mbox{e-}0  &  0.55  &  3.69\mbox{e-}0  &  0.50  \\
0 & 2  &  1.09\mbox{e-}1  &  0.90  &  6.40\mbox{e-}1  &  0.82  &  1.48\mbox{e-}0  &  0.71  &  2.28\mbox{e-}0  &  0.70  \\
 & 3  &  5.62\mbox{e-}2  &  0.96  &  3.40\mbox{e-}1  &  0.91  &  8.25\mbox{e-}1  &  0.84  &  1.28\mbox{e-}0  &  0.83  \\
 & 4  &  2.84\mbox{e-}2  &  0.98  &  1.75\mbox{e-}1  &  0.96  &  4.36\mbox{e-}1  &  0.92  &  6.78\mbox{e-}1  &  0.92  \\
\hline
 & 0  &  2.64\mbox{e-}3  &  --  &  2.75\mbox{e-}2  &  --  &  1.93\mbox{e-}1  &  --  &  5.04\mbox{e-}1  &  --  \\
 & 1  &  1.51\mbox{e-}4  &  4.12  &  1.03\mbox{e-}3  &  4.74  &  6.60\mbox{e-}3  &  4.87  &  1.20\mbox{e-}2  &  5.39  \\
1 & 2  &  8.95\mbox{e-}6  &  4.08  &  4.88\mbox{e-}5  &  4.39  &  2.56\mbox{e-}4  &  4.69  &  3.79\mbox{e-}4  &  4.98  \\
 & 3  &  5.42\mbox{e-}7  &  4.05  &  2.65\mbox{e-}6  &  4.20  &  1.18\mbox{e-}5  &  4.44  &  1.57\mbox{e-}5  &  4.60  \\
 & 4  &  3.33\mbox{e-}8  &  4.02  &  1.54\mbox{e-}7  &  4.10  &  6.17\mbox{e-}7  &  4.25  &  7.73\mbox{e-}7  &  4.34  \\
\hline
 & 0  &  5.40\mbox{e-}5  &  --  &  7.03\mbox{e-}4  &  --  &  1.06\mbox{e-}2  &  --  &  2.14\mbox{e-}2  &  --  \\
 & 1  &  8.04\mbox{e-}7  &  6.07  &  9.39\mbox{e-}6  &  6.23  &  1.43\mbox{e-}4  &  6.20  &  2.14\mbox{e-}4  &  6.64  \\
2 & 2  &  1.11\mbox{e-}8  &  6.18  &  1.09\mbox{e-}7  &  6.43  &  2.01\mbox{e-}6  &  6.16  &  2.73\mbox{e-}6  &  6.29  \\
 & 3  &  1.28\mbox{e-}10  &  6.44  &  1.32\mbox{e-}9  &  6.37  &  2.74\mbox{e-}8  &  6.20  &  3.34\mbox{e-}8  &  6.35  \\
 & 4  &  1.56\mbox{e-}12  &  6.35  &  2.19\mbox{e-}11  &  5.91  &  3.3\mbox{e-}10  &  6.38  &  2.88\mbox{e-}10  &  6.86  \\
\hline
 \end{array} $
}
\end{center}{$\phantom{|}$}
     \caption{Convergence of the postprocessed eigenvalues $\lambda_h^{\ast}$ for $\tau=1$.}
   \label{tab1b}
\end{table}
\begin{table}%[thbp]
  \begin{center}
  \scalebox{0.9}{%
    $\begin{array}{|c|c||c c  | c c||  c c|  c c|}
    \hline
    \multicolumn{2}{|c||}{\mbox{eigenmode}} & \multicolumn{4}{|c||}{\mbox{first}} &  \multicolumn{4}{|c|}{\mbox{fourth}} \\
    \hline
  \mbox{degree} & \mbox{mesh} & \multicolumn{2}{|c}{\|u-u_h\|_{\Th}} &  \multicolumn{2}{|c||}{\|u-u_h^\ast\|_{\Th}} & \multicolumn{2}{|c|}{\|u-u_h\|_{\Th}} & \multicolumn{2}{|c|}{\|u-u_h^\ast\|_{\Th}} \\
    k & \ell & \mbox{error}  & \mbox{order} & \mbox{error}  & \mbox{order} & \mbox{error}  & \mbox{order} & \mbox{error}  & \mbox{order} \\
 \hline
& 0  &  4.19\mbox{e-}2  &  --  &  8.65\mbox{e-}2  &  --  &  5.55\mbox{e-}1  &  --  &  6.11\mbox{e-}1  &  --  \\
 & 1  &  1.97\mbox{e-}2  &  1.09  &  2.72\mbox{e-}2  &  1.67  &  1.68\mbox{e-}1  &  1.73  &  1.91\mbox{e-}1  &  1.68  \\
0 & 2  &  9.46\mbox{e-}3  &  1.06  &  1.03\mbox{e-}2  &  1.40  &  6.83\mbox{e-}2  &  1.30  &  7.27\mbox{e-}2  &  1.40  \\
 & 3  &  4.62\mbox{e-}3  &  1.03  &  4.65\mbox{e-}3  &  1.15  &  3.09\mbox{e-}2  &  1.14  &  3.14\mbox{e-}2  &  1.21  \\
 & 4  &  2.28\mbox{e-}3  &  1.02  &  2.27\mbox{e-}3  &  1.04  &  1.47\mbox{e-}2  &  1.07  &  1.47\mbox{e-}2  &  1.09  \\
\hline
 & 0  &  5.02\mbox{e-}2  &  --  &  8.84\mbox{e-}3  &  --  &  3.21\mbox{e-}1  &  --  &  7.84\mbox{e-}2  &  --  \\
 & 1  &  1.26\mbox{e-}2  &  2.00  &  1.14\mbox{e-}3  &  2.96  &  8.10\mbox{e-}2  &  1.99  &  9.23\mbox{e-}3  &  3.09  \\
1 & 2  &  3.14\mbox{e-}3  &  2.00  &  1.44\mbox{e-}4  &  2.98  &  1.92\mbox{e-}2  &  2.07  &  1.13\mbox{e-}3  &  3.03  \\
 & 3  &  7.85\mbox{e-}4  &  2.00  &  1.82\mbox{e-}5  &  2.99  &  4.67\mbox{e-}3  &  2.04  &  1.39\mbox{e-}4  &  3.02  \\
 & 4  &  1.96\mbox{e-}4  &  2.00  &  2.28\mbox{e-}6  &  2.99  &  1.15\mbox{e-}3  &  2.02  &  1.72\mbox{e-}5  &  3.01  \\
\hline
 & 0  &  5.28\mbox{e-}3  &  --  &  8.36\mbox{e-}4  &  --  &  6.98\mbox{e-}2  &  --  &  1.31\mbox{e-}2  &  --  \\
 & 1  &  6.70\mbox{e-}4  &  2.98  &  5.29\mbox{e-}5  &  3.98  &  8.10\mbox{e-}3  &  3.11  &  7.26\mbox{e-}4  &  4.17  \\
2 & 2  &  8.41\mbox{e-}5  &  2.99  &  3.32\mbox{e-}6  &  4.00  &  9.78\mbox{e-}4  &  3.05  &  4.37\mbox{e-}5  &  4.06  \\
 & 3  &  1.05\mbox{e-}5  &  3.00  &  2.07\mbox{e-}7  &  4.00  &  1.20\mbox{e-}4  &  3.02  &  2.70\mbox{e-}6  &  4.02  \\
 & 4  &  1.32\mbox{e-}6  &  3.00  &  1.30\mbox{e-}8  &  4.00  &  1.49\mbox{e-}5  &  3.01  &  1.68\mbox{e-}7  &  4.00  \\
\hline
 \end{array} $
}
\end{center}{$\phantom{|}$}
\caption{Convergence  of the approximate ($u_h$) and postprocessed
  ($u_h^\ast$) eigenfunctions for $\tau=1$.}
   \label{tab1c}
\end{table}
\begin{table}%[thbp]
  \begin{center}
  \scalebox{0.9}{%
    $\begin{array}{|c|c||c c  | c c|  c c|  c c|}
    \hline
  \mbox{degree} & \mbox{mesh} & \multicolumn{2}{|c}{\mbox{first mode}} &  \multicolumn{2}{|c}{\mbox{second mode}} & \multicolumn{2}{|c|}{\mbox{fourth mode}} & \multicolumn{2}{|c|}{\mbox{sixth mode}} \\
    k & \ell & \mbox{error}  & \mbox{order} & \mbox{error}  & \mbox{order} & \mbox{error}  & \mbox{order} & \mbox{error}  & \mbox{order} \\
 \hline
 & 0  &  6.29\mbox{e-}1  &  --  &  2.83\mbox{e-}0  &  --  &  6.65\mbox{e-}0  &  --  &  8.95\mbox{e-}0  &  --  \\
 & 1  &  2.70\mbox{e-}1  &  1.22  &  1.40\mbox{e-}0  &  1.01  &  3.36\mbox{e-}0  &  0.99  &  4.72\mbox{e-}0  &  0.92  \\
0 & 2  &  1.25\mbox{e-}1  &  1.11  &  7.07\mbox{e-}1  &  0.99  &  1.72\mbox{e-}0  &  0.96  &  2.54\mbox{e-}0  &  0.89  \\
 & 3  &  5.99\mbox{e-}2  &  1.06  &  3.56\mbox{e-}1  &  0.99  &  8.85\mbox{e-}1  &  0.96  &  1.34\mbox{e-}0  &  0.92  \\
 & 4  &  2.93\mbox{e-}2  &  1.03  &  1.79\mbox{e-}1  &  0.99  &  4.50\mbox{e-}1  &  0.97  &  6.94\mbox{e-}1  &  0.95  \\
\hline
 & 0  &  6.21\mbox{e-}2  &  --  &  5.26\mbox{e-}1  &  --  &  1.85\mbox{e-}0  &  --  &  3.33\mbox{e-}0  &  --  \\
 & 1  &  1.41\mbox{e-}2  &  2.13  &  1.05\mbox{e-}1  &  2.33  &  3.38\mbox{e-}1  &  2.45  &  6.01\mbox{e-}1  &  2.47  \\
1 & 2  &  3.37\mbox{e-}3  &  2.07  &  2.30\mbox{e-}2  &  2.19  &  6.70\mbox{e-}2  &  2.33  &  1.12\mbox{e-}1  &  2.42  \\
 & 3  &  8.23\mbox{e-}4  &  2.03  &  5.37\mbox{e-}3  &  2.10  &  1.47\mbox{e-}2  &  2.19  &  2.39\mbox{e-}2  &  2.24  \\
 & 4  &  2.03\mbox{e-}4  &  2.02  &  1.30\mbox{e-}3  &  2.05  &  3.44\mbox{e-}3  &  2.10  &  5.48\mbox{e-}3  &  2.12  \\
\hline
 & 0  &  3.23\mbox{e-}2  &  --  &  2.09\mbox{e-}1  &  --  &  5.86\mbox{e-}1  &  --  &  9.78\mbox{e-}1  &  --  \\
 & 1  &  8.11\mbox{e-}3  &  1.99  &  5.10\mbox{e-}2  &  2.03  &  1.32\mbox{e-}1  &  2.15  &  2.09\mbox{e-}1  &  2.23  \\
2 & 2  &  2.04\mbox{e-}3  &  1.99  &  1.28\mbox{e-}2  &  2.00  &  3.28\mbox{e-}2  &  2.02  &  5.13\mbox{e-}2  &  2.03  \\
 & 3  &  5.12\mbox{e-}4  &  1.99  &  3.20\mbox{e-}3  &  2.00  &  8.20\mbox{e-}3  &  2.00  &  1.28\mbox{e-}2  &  2.00  \\
 & 4  &  1.28\mbox{e-}4  &  2.00  &  8.02\mbox{e-}4  &  2.00  &  2.05\mbox{e-}3  &  2.00  &  3.21\mbox{e-}3  &  2.00  \\
\hline
\end{array} $
}
\end{center}{$\phantom{|}$}
     \caption{Convergence of \mbox{$|\lambda_h-\lt_h|$} to 0 for $\tau=1$.}
   \label{tab2}
\end{table}
\begin{table}%[thbp]
  \begin{center}
  \scalebox{0.9}{%
    $\begin{array}{|c|c||c c  | c c|  c c|  c c|}
    \hline
  \mbox{degree} & \mbox{mesh} & \multicolumn{2}{|c}{\mbox{first mode}} &  \multicolumn{2}{|c}{\mbox{second mode}} & \multicolumn{2}{|c|}{\mbox{fourth mode}} & \multicolumn{2}{|c|}{\mbox{sixth mode}} \\
    k & \ell & \mbox{error}  & \mbox{order} & \mbox{error}  & \mbox{order} & \mbox{error}  & \mbox{order} & \mbox{error}  & \mbox{order} \\
 \hline
 & 0  &  1.42\mbox{e-}1  &  --  &  1.72\mbox{e-}0  &  --  &  3.76\mbox{e-}0  &  --  &  5.45\mbox{e-}0  &  --  \\
 & 1  &  3.74\mbox{e-}1  &  -1.39  &  2.01\mbox{e-}0  &  -0.22  &  4.15\mbox{e-}0  &  -0.14  &  5.82\mbox{e-}0  &  -0.09  \\
0 & 2  &  4.33\mbox{e-}1  &  -0.21  &  2.09\mbox{e-}0  &  -0.05  &  4.27\mbox{e-}0  &  -0.04  &  5.91\mbox{e-}0  &  -0.02  \\
 & 3  &  4.48\mbox{e-}1  &  -0.05  &  2.11\mbox{e-}0  &  -0.01  &  4.30\mbox{e-}0  &  -0.01  &  5.93\mbox{e-}0  &  -0.01  \\
 & 4  &  4.52\mbox{e-}1  &  -0.01  &  2.11\mbox{e-}0  &  0.00  &  4.31\mbox{e-}0  &  0.00 &  5.94\mbox{e-}0  &  -0.00 \\
\hline
 & 0  &  1.29\mbox{e-}2  &  --  &  2.45\mbox{e-}1  &  --  &  1.31\mbox{e-}0  &  --  &  2.47\mbox{e-}0  &  --  \\
 & 1  &  4.44\mbox{e-}3  &  1.54  &  9.15\mbox{e-}2  &  1.42  &  5.57\mbox{e-}1  &  1.23  &  1.24\mbox{e-}0  &  0.99  \\
1 & 2  &  1.38\mbox{e-}3  &  1.69  &  2.55\mbox{e-}2  &  1.84  &  1.76\mbox{e-}1  &  1.66  &  4.64\mbox{e-}1  &  1.42  \\
 & 3  &  3.62\mbox{e-}4  &  1.93  &  6.56\mbox{e-}3  &  1.96  &  4.73\mbox{e-}2  &  1.90  &  1.37\mbox{e-}1  &  1.76  \\
 & 4  &  9.15\mbox{e-}5  &  1.98  &  1.65\mbox{e-}3  &  1.99  &  1.20\mbox{e-}2  &  1.97  &  3.60\mbox{e-}2  &  1.93  \\
\hline
 & 0  &  1.81\mbox{e-}4  &  --  &  6.85\mbox{e-}3  &  --  &  7.60\mbox{e-}2  &  --  &  2.19\mbox{e-}1  &  --  \\
 & 1  &  1.57\mbox{e-}5  &  3.53  &  6.21\mbox{e-}4  &  3.46  &  7.18\mbox{e-}3  &  3.40  &  1.89\mbox{e-}2  &  3.54  \\
2 & 2  &  1.25\mbox{e-}6  &  3.65  &  4.24\mbox{e-}5  &  3.87  &  4.93\mbox{e-}4  &  3.87  &  1.26\mbox{e-}3  &  3.90  \\
 & 3  &  8.23\mbox{e-}8  &  3.92  &  2.70\mbox{e-}6  &  3.97  &  3.15\mbox{e-}5  &  3.97  &  8.01\mbox{e-}5  &  3.98  \\
 & 4  &  5.21\mbox{e-}9  &  3.98  &  1.70\mbox{e-}7  &  3.99  &  1.98\mbox{e-}6  &  3.99  &  5.02\mbox{e-}6  &  3.99  \\
\hline
 \end{array} $
}
\end{center}{$\phantom{|}$}
     \caption{Convergence  of the approximate eigenvalues $\lambda_h$ for $\tau=h$.}
   \label{tab3}
\end{table}
\begin{table}%[thbp]
  \begin{center}
  \scalebox{0.9}{%
    $\begin{array}{|c|c||c c  | c c|  c c|  c c|}
    \hline
  \mbox{degree} & \mbox{mesh} & \multicolumn{2}{|c}{\mbox{first mode}} &  \multicolumn{2}{|c}{\mbox{second mode}} & \multicolumn{2}{|c|}{\mbox{fourth mode}} & \multicolumn{2}{|c|}{\mbox{sixth mode}} \\
    k & \ell & \mbox{error}  & \mbox{order} & \mbox{error}  & \mbox{order} & \mbox{error}  & \mbox{order} & \mbox{error}  & \mbox{order} \\
 \hline
& 0  &  2.53\mbox{e-}1  &  --  &  7.42\mbox{e-}1  &  --  &  2.19\mbox{e-}0  &  --  &  3.57\mbox{e-}0  &  --  \\
 & 1  &  5.73\mbox{e-}1  &  -1.18  &  7.76\mbox{e-}1  &  -0.07  &  8.88\mbox{e-}1  &  1.30  &  6.80\mbox{e-}1  &  2.39  \\
0 & 2  &  6.68\mbox{e-}1  &  -0.22  &  1.33\mbox{e-}0  &  -0.78  &  2.18\mbox{e-}0  &  -1.29  &  2.75\mbox{e-}0  &  -2.01  \\
 & 3  &  6.92\mbox{e-}1  &  -0.05  &  1.49\mbox{e-}0  &  -0.16  &  2.56\mbox{e-}0  &  -0.23  &  3.39\mbox{e-}0  &  -0.30  \\
 & 4  &  6.99\mbox{e-}1  &  -0.01  &  1.53\mbox{e-}0  &  -0.04  &  2.66\mbox{e-}0  &  -0.05  &  3.56\mbox{e-}0  &  -0.07  \\
\hline
 & 0  &  1.46\mbox{e-}2  &  --  &  5.70\mbox{e-}2  &  --  &  3.88\mbox{e-}1  &  --  &  2.92\mbox{e-}1  &  --  \\
 & 1  &  5.84\mbox{e-}3  &  1.32  &  1.95\mbox{e-}2  &  1.55  &  5.78\mbox{e-}2  &  2.75  &  9.51\mbox{e-}2  &  1.62  \\
1 & 2  &  1.59\mbox{e-}3  &  1.88  &  6.75\mbox{e-}3  &  1.53  &  2.33\mbox{e-}2  &  1.31  &  3.05\mbox{e-}2  &  1.64  \\
 & 3  &  4.06\mbox{e-}4  &  1.97  &  1.80\mbox{e-}3  &  1.91  &  6.36\mbox{e-}3  &  1.87  &  8.70\mbox{e-}3  &  1.81  \\
 & 4  &  1.02\mbox{e-}4  &  1.99  &  4.57\mbox{e-}4  &  1.98  &  1.62\mbox{e-}3  &  1.97  &  2.24\mbox{e-}3  &  1.96  \\
\hline
 & 0  &  2.78\mbox{e-}4  &  --  &  1.12\mbox{e-}3  &  --  &  1.76\mbox{e-}2  &  --  &  2.47\mbox{e-}2  &  --  \\
 & 1  &  2.55\mbox{e-}5  &  3.45  &  1.82\mbox{e-}4  &  2.62  &  1.11\mbox{e-}3  &  3.99  &  2.10\mbox{e-}3  &  3.57  \\
2 & 2  &  1.72\mbox{e-}6  &  3.89  &  1.51\mbox{e-}5  &  3.59  &  1.02\mbox{e-}4  &  3.44  &  1.42\mbox{e-}4  &  3.87  \\
 & 3  &  1.10\mbox{e-}7  &  3.97  &  1.00\mbox{e-}6  &  3.91  &  6.89\mbox{e-}6  &  3.89  &  9.97\mbox{e-}6  &  3.83  \\
 & 4  &  6.88\mbox{e-}9  &  3.99  &  6.34\mbox{e-}8  &  3.98  &  4.39\mbox{e-}7  &  3.97  &  6.40\mbox{e-}7  &  3.96  \\
\hline
 \end{array} $
}
\end{center}{$\phantom{|}$}
     \caption{Convergence  of the approximate eigenvalues $\lambda_h$ for $\tau=1/h$.}
   \label{tab4}
\end{table}
}}

We consider the domain $\Omega = (0, \pi) \times (0,\pi)$. In this
case, the exact eigenvalues and eigenfunctions are given by
$\lambda^{mn} = m^2 + n^2$ and $u^{mn} = \sin(m x) \sin (nx)$,
respectively, for $m, n \in \mathbb{N}_{+}$. Clearly, the
eigenfunctions are infinitely smooth, so the convergence rates should
be limited only by the degrees of the approximating polynomials.  We
obtain an initial mesh by subdividing $\Omega$ into a uniform grid of
$4 \times 4$ squares and splitting each square into two triangles by
its positively sloped diagonal. Successively finer meshes are obtained
by subdividing each triangle into four smaller triangles. The mesh at
``level $\ell$'' is obtained from the initial mesh by $\ell$
refinements. We compute the solution of the HDG eigenproblem in each
of these meshes. The results obtained are collected below.

In Table~\ref{tab1}, the error and order of convergence of the
approximate eigenvalues for $\tau=1$ is presented. We see that the
approximate eigenvalues $\lambda_h$ converge to the exact values at
the rate of $O(h^{2k+1})$. This is in good agreement with the
theoretical result of Theorem~\ref{thm:ewrate}.  To compare with the
HRT result, see~\cite[Table~1]{hrt:2010}, which shows that the HRT
eigenvalues converge faster at the rate $O(h^{2k+2})$.

Curiously however, we found that the simple local
postprocessing~\eqref{eq:post} can give eigenvalues competitive with
the HRT eigenvalues.  In Table~\ref{tab1b}, we present the error and
order of convergence of the postprocessed eigenvalues
$\lambda_h^\ast$.  There is no change in the convergence order for
$k=0$ as both the approximate and postprocessed eigenvalues converge
linearly.  However, observe that when $k\ge 1$, the postprocessed
eigenvalues converge at a faster rate of $O(h^{2k+2})$.  Presently, we
do not have a rigorous proof for this convergence rate.  This seems to
be the fastest rate we can expect. Indeed, we also observed, in unreported experiments, that when the
postprocessing~\eqref{eq:post} is applied to the HRT method,
no further improvement in the  convergence rate beyond the
$O(h^{2k+2})$-rate was obtained.
 To conclude the discussion on the postprocessing,
see Table~\ref{tab1c}, where the error and order of convergence of the
approximate and postprocessed eigenfunctions is presented. We see that
the convergence rate of the approximate eigenfunctions is
$O(h^{k+1})$, while the convergence rate of the postprocessed
eigenfunctions is $O(h^{k+2})$ for $k \ge 1$ and $O(h^{k+1})$ for
$k=0$. These results illustrate that the local postprocessing is
effective for $k \ge 1$ as it increases the convergence rate of both
the eigenvalue and eigenfunction approximations by one order.

In Table~\ref{tab2}, we display the absolute value of difference
between the approximate eigenvalues $\lambda_h$ and the perturbed
eigenvalues $\lt_h$ of Section~5, namely $|\lambda_h - \lt_h|$. We see
that the numbers $\lt_h$ can serve as good approximations of
$\lambda_h$.  The difference $|\lambda_h - \tilde{\lambda}_h|$ is seen
to decrease with $h$ at the rate $O(h)$ for $k=0$.  This is in
accordance with the estimate of Theorem~\ref{thm:reduction:eig}.
However, the same difference decreases at the rate $O(h^2)$ for $k \ge
1$, which is one order higher than that predicted by
Theorem~\ref{thm:reduction:eig}. These results show that by solving a
standard generalized eigenproblem for $\tilde\lambda_h$, we obtain
very effective initial iterates for nonlinear iterative algorithms to
solve the condensed nonlinear eigenproblem for $\lambda_h$.  At the
same time, we note that $\tilde\lambda_h$ by itself does not converge to the
exact eigenvalue $\lambda$ at as fast a rate as $\lambda_h$.

Next, we examine the performance of the HDG method for different
choices of $\tau$. In particular, we show the error and order of
convergence of the approximate eigenvalues for $\tau=h$ in
Table~\ref{tab3} and for $\tau=1/h$ in Table~\ref{tab4}. We observe in
both cases that the approximate eigenvalues converge at the rate
$O(h^{2k})$. This convergence rate is one order less than the
convergence rate of the approximate eigenvalues for $\tau=1$.  A
theoretical explanation for this phenomena follows
from~\eqref{eq:PIapprx}: When either $\tau=h$ or $\tau = 1/h$, one of
the bounds in~\eqref{eq:PIapprx} deteriorate by one order in $h$. When
these revised estimates are used in the ensuing eigenvalue convergence
analysis, we obtain the reduced $O(h^{2k})$-rate.

\subsection{L-shaped domain}

To study the limitations imposed by singularities of eigenfunctions,
we consider the L-shaped domain $\Omega = \Omega_0 \backslash
\Omega_1$, where $\Omega_0 \equiv (0,2) \times (0,2)$ and $\Omega_1
\equiv (1,2) \times (1,2)$ are the square domains.
This domain has both singular and smooth eigenfunctions, so offers an
interesting example to study the changes in convergence rates due to
singularities.
As before, we consider triangular meshes that are successive uniform refinements of an initial uniform mesh. The initial mesh is obtained as in the previous example using a $4 \times 4$ uniform grid of $\Omega_0$, except we now omit all triangles in $\Omega_1$.

Since $\Omega$ has a reentrant corner at the point $(1,1)$, some eigenfunctions are
singular. Specifically, we may only expect Theorem 4.1 to hold
with $s_{\lambda} = \frac{2}{3}-\varepsilon$ for an arbitrarily small
$\varepsilon>0$ for singular eigenfunctions. For this L-shaped domain the first eigenmode is singular and the corresponding eigenvalue is calculated in \cite{BetckeTrefethen05} as $\lambda_1 = 9.63972384464540$. It is interesting to note that the third eigenmode is smooth and the third eigenvalue is known exactly as $\lambda_3 = 2\pi^2$.

The errors and resulting order of convergence for the approximate and
postprocessed eigenvalues are reported in Table~\ref{tab5} for the
first and third eigenmodes. We observe that the convergence rate of
the approximate smallest eigenvalue is at most $O(h^{4/3})$ which
agrees with the {\it a priori} error estimate given by
Theorem~4.1. Furthermore, the postprocessed smallest eigenvalue also
converges at the same order 4/3 for $k \ge 1$. However, for the
third eigenmode, the approximate eigenvalue converges at order
$O(h^{2k+1})$ and the postprocessed eigenvalue converges at  order
$O(h^{2k+2})$ for $k \ge 1$.

{\footnotesize{
\begin{table}%[thbp]
  \begin{center}
  \scalebox{0.9}{%
    $\begin{array}{|c|c||c c  | c c||  c c|  c c|}
    \hline
    \multicolumn{2}{|c||}{\mbox{eigenmode}} & \multicolumn{4}{|c||}{\mbox{first}} &  \multicolumn{4}{|c|}{\mbox{third}} \\
    \hline
  \mbox{degree} & \mbox{mesh} & \multicolumn{2}{|c}{|\lambda-\lambda_h|} &  \multicolumn{2}{|c||}{|\lambda-\lambda_h^\ast|} & \multicolumn{2}{|c|}{|\lambda - \lambda_h|} & \multicolumn{2}{|c|}{|\lambda - \lambda_h^\ast|} \\
    k & \ell & \mbox{error}  & \mbox{order} & \mbox{error}  & \mbox{order} & \mbox{error}  & \mbox{order} & \mbox{error}  & \mbox{order} \\
 \hline
 & 0  &  3.48\mbox{e-}0  &  --  &  4.20\mbox{e-}0  &  --  &  1.06\mbox{e-}9  &  --  &  1.12\mbox{e-}9  &  --  \\
 & 1  &  2.12\mbox{e-}0  &  0.72  &  2.71\mbox{e-}0  &  0.63  &  7.39\mbox{e-}0  &  0.52  &  8.10\mbox{e-}0  &  0.46  \\
0 & 2  &  1.18\mbox{e-}0  &  0.84  &  1.56\mbox{e-}0  &  0.80  &  4.64\mbox{e-}0  &  0.67  &  5.21\mbox{e-}0  &  0.64  \\
 & 3  &  6.24\mbox{e-}1  &  0.92  &  8.34\mbox{e-}1  &  0.90  &  2.66\mbox{e-}0  &  0.80  &  3.02\mbox{e-}0  &  0.79  \\
 & 4  &  3.19\mbox{e-}1  &  0.97  &  4.29\mbox{e-}1  &  0.96  &  1.44\mbox{e-}0  &  0.89  &  1.64\mbox{e-}0  &  0.88  \\
\hline
 & 0  &  5.04\mbox{e-}1  &  --  &  1.24\mbox{e-}1  &  --  &  4.04\mbox{e-}0  &  --  &  7.98\mbox{e-}1  &  --  \\
 & 1  &  1.02\mbox{e-}1  &  2.31  &  5.93\mbox{e-}2  &  1.07  &  5.68\mbox{e-}1  &  2.83  &  4.43\mbox{e-}2  &  4.17  \\
1 & 2  &  2.82\mbox{e-}2  &  1.85  &  2.36\mbox{e-}2  &  1.33  &  6.19\mbox{e-}2  &  3.20  &  2.35\mbox{e-}3  &  4.24  \\
 & 3  &  9.85\mbox{e-}3  &  1.52  &  9.37\mbox{e-}3  &  1.33  &  7.10\mbox{e-}3  &  3.12  &  1.39\mbox{e-}4  &  4.08  \\
 & 4  &  3.77\mbox{e-}3  &  1.39  &  3.73\mbox{e-}3  &  1.33  &  8.50\mbox{e-}4  &  3.06  &  8.15\mbox{e-}6  &  4.09  \\
\hline
 & 0  &  6.93\mbox{e-}2  &  --  &  5.78\mbox{e-}2  &  --  &  3.03\mbox{e-}1  &  --  &  3.91\mbox{e-}2  &  --  \\
 & 1  &  2.35\mbox{e-}2  &  1.56  &  2.36\mbox{e-}2  &  1.29  &  7.73\mbox{e-}3  &  5.29  &  3.48\mbox{e-}4  &  6.81  \\
2 & 2  &  9.32\mbox{e-}3  &  1.33  &  9.41\mbox{e-}3  &  1.33  &  2.19\mbox{e-}4  &  5.14  &  4.48\mbox{e-}6  &  6.28  \\
 & 3  &  3.73\mbox{e-}3  &  1.32  &  3.75\mbox{e-}3  &  1.33  &  6.50\mbox{e-}6  &  5.07  &  6.13\mbox{e-}8  &  6.19  \\
 & 4  &  1.49\mbox{e-}3  &  1.33  &  1.49\mbox{e-}3  &  1.33  &  1.98\mbox{e-}7  &  5.04  &  7.9\mbox{e-}10  &  6.28  \\
\hline
 \end{array} $
}
\end{center}{$\phantom{|}$}
\caption{Convergence  of the approximate ($\lambda_h$) and postprocessed
  ($\lambda_h^\ast$) eigenvalues for the L-shaped domain problem.}
   \label{tab5}
\end{table}
}}

\subsection{Other polynomial spaces}
\label{ssec:extension}

So far the presentation focused on the case when the same polynomial
degree $k$ is employed to approximate the solution, the flux, and the
trace. It is interesting to examine the case of mixed degrees. Since
$M_h$ determines the size of the global system, let us hold $M_h$
fixed as set in~\eqref{eq:M} consisting of functions of degree $k$ on
the element interfaces, while varying the degrees of $W_h$ and $V_h$
as follows:

Case 1: $[\tau]_K\ge 0$ on $\d K$ for all $K\in\Th$, $k\ge 1$,  and
\begin{align*}
%  \label{eq:W:c2}
  W_h &= \{ w: \text{ for every mesh element }K, w|_K \in P_{k-1}(K)\},
  \\
%\label{eq:V:c2}
  V_h &=  \{ \vv: \text{ for every mesh element }K, \vv|_K \in P_k(K)^n\}.
\end{align*}

Case 2: $[\tau]_K> 0$ on $\d K$ for all $K\in\Th$, $k\ge 1$, and
\begin{align*}
 % \label{eq:W:c3}
  W_h &= \{ w: \text{ for every mesh element }K, w|_K \in P_{k}(K)\},
  \\
  % \label{eq:V:c3}
  V_h &=  \{ \vv: \text{ for every mesh element }K, \vv|_K \in P_{k-1}(K)^n\}.
\end{align*}

These cases are interesting because the source problem~\eqref{method}
is uniquely solvable~\cite[see eq.~(3.5) and
Proposition~3.2]{CockbGopalLazar09}.  One can now follow the same
techniques in~\cite[Lemma 3.2 and Theorem 3.1]{CockbDuboiGopal10} to
obtain a bound for the norm of the operator $\Uw$, leading to results
analogous to Theorem~\ref{thm:reduction:eig} for the reduced
eigenproblem in both these cases.  On the other hand, rigorous proofs
of eigenvalue convergence rates for general $\tau$ are yet to be
developed for these cases: Ingredients (of Section~\ref{sec:rate})
that need generalization include the projection
satisfying~\eqref{eq:PIapprx} and the $\tau$-explicit
estimates~\eqref{eq:ee-u}--\eqref{eq:ee-q}.

Nevertheless, numerical results are not encouraging in these cases.
Returning to the eigenproblem on the square described in
\S\ref{ssec:square}, the convergence of approximate and postprocessed
eigenvalues in both cases, obtained with $\tau=1$, are presented in
Tables~\ref{tab4b}--\ref{tab4e}.  We observe that the approximate
eigenvalues converge at the rate $O(h^{2k-1})$ in both cases.  It also
appears that the postprocessed eigenvalues converge at the rate
$O(h^{2k})$, except in Case~2 with $k=1$, where $\lambda_h^\ast$
converges at the same rate as $\lambda_h$. Comparing the results from
approximation spaces of mixed and equal degree, we find that the
equal-degree spaces give two orders faster convergence rate, so is
clearly preferable in the eigenvalue context.

The special case of $\tau \equiv 0$ deserves further remarks. In this
case, if $V_h$ is changed to the (larger) piecewise Raviart-Thomas
space, then the resulting HDG eigenproblem is the same as the
hybridized Raviart-Thomas eigenproblem. This case is fully studied
in~\cite{hrt:2010}. Next, consider the special case of $\tau \equiv 0$
in Case~1. Then, the HDG formulation reduces to the (hybridized) mixed
Brezzi-Douglas-Marini (BDM) eigenproblem, for which a complete
convergence theory is available in~\cite{BoffiBrezzGasta00}.  In
particular, it follows from their results that the BDM eigenvalue
convergence rate is $O(h^{2k})$. However, the postprocessing of
\S\ref{ssec:postprocessing} provides a way to make the BDM eigenvalues
more competitive: We observe, in Table \ref{tab4f}, that postprocessed
BDM eigenvalues converge at $O(h^{2k+2}),$ the rate also observed for
the HRT and HDG eigenvalues after the same postprocessing.

We conclude with a summary of the convergence rates in
Table~\ref{tab:summary}. Its entries are based on the observed and
known convergence rates in various cases for the first eigenpair of
the square domain example in \S\ref{ssec:square}. Note that, as
before, the convergence of all functions are measured in the
$L^2(\Omega)$-norm.

{\footnotesize{
\begin{table}[htbp]
  \centering
  \begin{tabular}{|l||ccccc|}
    \hline
      \multicolumn{1}{|c||}{Method}
      &
      \multicolumn{5}{c|}{Convergence rates}
      \\ \cline{2-6}
      &$\lambda_h$&$\lambda_h^*$& $\vq_h$ & $u_h$  & $u_h^*$
      \\ \hline
      HDG, equal degree $k\ge 0, \tau=1$
      & $2k+1$   &$2k+2$       &   $k+1$   &  $k+1$ & $k+2-\delta_{k0}$
      \\
      HDG, equal degree $k\ge 1, \tau=h$
      & $2k$     &     $ 2k+2$      &    $k+1$   &    $k$    & $k+2$
      \\
      HDG, equal degree $k\ge 1, \tau=1/h$
      & $2k$     &      $2k$       &    $k$   &   $k+1$   &  $k+1$
      \\
      HDG, Case~1, $k\ge 1$, $\tau=1$
      & $2k-1$   &  $2k$       &    $k$   &   $k$    &  $k+1$
      \\
      HDG, Case~2, $k\ge 1$, $\tau=1$
      & $2k-1$ &$2k-\delta_{k1}$&    $k$   &  $k$     & $k+1$
      \\
      Interior penalty DG~\cite{AntonBuffaPerug06}, $k \ge 1$
      & $2k$   & --             & --   & $k+1$        & --
      \\
     HRT~\cite{hrt:2010} ($k \ge 0$, $\tau = 0$)
      & $2k+2$   &$2k+2$       &   $k+1$    &  $k+1$     &  $k+2$
      \\
      BDM  (Case~1, $k \ge 1$, $\tau = 0$)
      & $2k$ & $2k+2$       &  $k+1$      &  $k$      & $k+2$
      \\
      \hline
  \end{tabular}
    \vspace{0.2in}
  \caption{Summary of convergence rates for smooth eigenfunctions}
  \label{tab:summary}
\end{table}
}}

{\footnotesize{
\begin{table}%[thbp]
  \begin{center}
  \scalebox{0.9}{%
    $\begin{array}{|c|c||c c  | c c|  c c|  c c|}
    \hline
  \mbox{degree} & \mbox{mesh} & \multicolumn{2}{|c}{\mbox{first mode}} &  \multicolumn{2}{|c}{\mbox{second mode}} & \multicolumn{2}{|c|}{\mbox{fourth mode}} & \multicolumn{2}{|c|}{\mbox{sixth mode}} \\
    k & \ell & \mbox{error}  & \mbox{order} & \mbox{error}  & \mbox{order} & \mbox{error}  & \mbox{order} & \mbox{error}  & \mbox{order} \\
\hline
& 0  &  1.12\mbox{e-}0  &  --  &  2.68\mbox{e-}0  &  --  &  4.90\mbox{e-}0  &  --  &  7.05\mbox{e-}0  &  --  \\
 & 1  &  5.78\mbox{e-}1  &  0.96  &  1.33\mbox{e-}0  &  1.01  &  2.50\mbox{e-}0  &  0.97  &  3.64\mbox{e-}0  &  0.95  \\
1 & 2  &  2.93\mbox{e-}1  &  0.98  &  6.53\mbox{e-}1  &  1.03  &  1.23\mbox{e-}0  &  1.03  &  1.67\mbox{e-}0  &  1.12  \\
 & 3  &  1.48\mbox{e-}1  &  0.99  &  3.23\mbox{e-}1  &  1.02  &  6.05\mbox{e-}1  &  1.02  &  7.91\mbox{e-}1  &  1.08  \\
 & 4  &  7.41\mbox{e-}2  &  0.99  &  1.60\mbox{e-}1  &  1.01  &  3.00\mbox{e-}1  &  1.01  &  3.84\mbox{e-}1  &  1.04  \\
\hline
 & 0  &  3.41\mbox{e-}2  &  --  &  1.54\mbox{e-}1  &  --  &  5.39\mbox{e-}1  &  --  &  8.13\mbox{e-}1  &  --  \\
 & 1  &  4.72\mbox{e-}3  &  2.85  &  2.13\mbox{e-}2  &  2.85  &  7.93\mbox{e-}2  &  2.76  &  1.12\mbox{e-}1  &  2.85  \\
2 & 2  &  6.18\mbox{e-}4  &  2.93  &  2.76\mbox{e-}3  &  2.95  &  1.03\mbox{e-}2  &  2.94  &  1.43\mbox{e-}2  &  2.98  \\
 & 3  &  7.90\mbox{e-}5  &  2.97  &  3.49\mbox{e-}4  &  2.98  &  1.30\mbox{e-}3  &  2.99  &  1.78\mbox{e-}3  &  3.00  \\
 & 4  &  9.99\mbox{e-}6  &  2.98  &  4.38\mbox{e-}5  &  2.99  &  1.62\mbox{e-}4  &  3.00  &  2.22\mbox{e-}4  &  3.01  \\
\hline
 & 0  &  5.81\mbox{e-}4  &  --  &  4.88\mbox{e-}3  &  --  &  3.56\mbox{e-}2  &  --  &  5.23\mbox{e-}2  &  --  \\
 & 1  &  1.99\mbox{e-}5  &  4.87  &  1.76\mbox{e-}4  &  4.80  &  1.31\mbox{e-}3  &  4.76  &  1.87\mbox{e-}3  &  4.80  \\
3 & 2  &  6.44\mbox{e-}7  &  4.95  &  5.68\mbox{e-}6  &  4.95  &  4.25\mbox{e-}5  &  4.95  &  5.97\mbox{e-}5  &  4.97  \\
 & 3  &  2.04\mbox{e-}8  &  4.98  &  1.79\mbox{e-}7  &  4.99  &  1.33\mbox{e-}6  &  4.99  &  1.86\mbox{e-}6  &  5.00  \\
 & 4  &  6.43\mbox{e-}10  &  4.99  &  1.13\mbox{e-}8  &  3.99  &  4.16\mbox{e-}8  &  5.00  &  5.78\mbox{e-}8  &  5.01  \\
\hline
 \end{array} $
}
\end{center}{$\phantom{|}$}
     \caption{{Convergence  of the approximate eigenvalues $\lambda_h$ for (the mixed-degree) Case 1 with $\tau = 1$.% The polynomial degree $k$ is used for $\vec{q}_h$ and $\widehat{u}_h$, while $k-1$ for $u_h$.
       }}
   \label{tab4b}
\end{table}
\begin{table}%[thbp]
  \begin{center}
  \scalebox{0.9}{%
    $\begin{array}{|c|c||c c  | c c|  c c|  c c|}
    \hline
  \mbox{degree} & \mbox{mesh} & \multicolumn{2}{|c}{\mbox{first mode}} &  \multicolumn{2}{|c}{\mbox{second mode}} & \multicolumn{2}{|c|}{\mbox{fourth mode}} & \multicolumn{2}{|c|}{\mbox{sixth mode}} \\
    k & \ell & \mbox{error}  & \mbox{order} & \mbox{error}  & \mbox{order} & \mbox{error}  & \mbox{order} & \mbox{error}  & \mbox{order} \\
\hline
 & 0  &  1.20\mbox{e-}1  &  --  &  3.54\mbox{e-}1  &  --  &  7.98\mbox{e-}1  &  --  &  1.06\mbox{e-}0  &  --  \\
 & 1  &  3.74\mbox{e-}2  &  1.68  &  9.45\mbox{e-}2  &  1.90  &  2.57\mbox{e-}1  &  1.64  &  1.98\mbox{e-}1  &  2.43  \\
1 & 2  &  1.07\mbox{e-}2  &  1.81  &  2.51\mbox{e-}2  &  1.91  &  7.70\mbox{e-}2  &  1.74  &  3.22\mbox{e-}2  &  2.62  \\
 & 3  &  2.88\mbox{e-}3  &  1.89  &  6.55\mbox{e-}3  &  1.94  &  2.11\mbox{e-}2  &  1.86  &  6.26\mbox{e-}3  &  2.36  \\
 & 4  &  7.51\mbox{e-}4  &  1.94  &  1.68\mbox{e-}3  &  1.96  &  5.52\mbox{e-}3  &  1.94  &  1.39\mbox{e-}3  &  2.17  \\
\hline
 & 0  &  5.98\mbox{e-}3  &  --  &  2.80\mbox{e-}2  &  --  &  1.13\mbox{e-}1  &  --  &  1.50\mbox{e-}1  &  --  \\
 & 1  &  4.86\mbox{e-}4  &  3.62  &  2.22\mbox{e-}3  &  3.65  &  8.95\mbox{e-}3  &  3.66  &  1.21\mbox{e-}2  &  3.63  \\
2 & 2  &  3.46\mbox{e-}5  &  3.81  &  1.54\mbox{e-}4  &  3.85  &  6.03\mbox{e-}4  &  3.89  &  8.41\mbox{e-}4  &  3.85  \\
 & 3  &  2.31\mbox{e-}6  &  3.91  &  1.00\mbox{e-}5  &  3.94  &  3.87\mbox{e-}5  &  3.96  &  5.40\mbox{e-}5  &  3.96  \\
 & 4  &  1.49\mbox{e-}7  &  3.95  &  6.39\mbox{e-}7  &  3.97  &  2.44\mbox{e-}6  &  3.99  &  3.39\mbox{e-}6  &  3.99  \\
\hline
 & 0  &  7.95\mbox{e-}5  &  --  &  6.94\mbox{e-}4  &  --  &  6.15\mbox{e-}3  &  --  &  8.36\mbox{e-}3  &  --  \\
 & 1  &  1.42\mbox{e-}6  &  5.80  &  1.25\mbox{e-}5  &  5.79  &  1.05\mbox{e-}4  &  5.87  &  1.46\mbox{e-}4  &  5.84  \\
3 & 2  &  2.37\mbox{e-}8  &  5.91  &  2.04\mbox{e-}7  &  5.94  &  1.64\mbox{e-}6  &  6.00  &  2.24\mbox{e-}6  &  6.02  \\
 & 3  &  3.77\mbox{e-}10  &  5.97  &  3.23\mbox{e-}9  &  5.98  &  2.55\mbox{e-}8  &  6.01  &  3.45\mbox{e-}8  &  6.02  \\
 & 4  &  6.08\mbox{e-}12  &  5.95  &  5.04\mbox{e-}11  &  6.00  &  3.77\mbox{e-}10  &  6.08  &  5.14\mbox{e-}10  &  6.07  \\
\hline
 \end{array} $
}
\end{center}{$\phantom{|}$}
     \caption{{Convergence  of the postprocessed eigenvalues $\lambda_h^\ast$ for (the mixed-degree) Case 1 with $\tau = 1$. % The polynomial degree $k$ is used for $\vec{q}_h$ and $\widehat{u}_h$, while $k-1$ for $u_h$.
       }}
   \label{tab4c}
\end{table}
\begin{table}%[thbp]
  \begin{center}
  \scalebox{0.9}{%
    $\begin{array}{|c|c||c c  | c c|  c c|  c c|}
    \hline
  \mbox{degree} & \mbox{mesh} & \multicolumn{2}{|c}{\mbox{first mode}} &  \multicolumn{2}{|c}{\mbox{second mode}} & \multicolumn{2}{|c|}{\mbox{fourth mode}} & \multicolumn{2}{|c|}{\mbox{sixth mode}} \\
    k & \ell & \mbox{error}  & \mbox{order} & \mbox{error}  & \mbox{order} & \mbox{error}  & \mbox{order} & \mbox{error}  & \mbox{order} \\
 \hline
& 0  &  3.89\mbox{e-}1  &  --  &  1.92\mbox{e-}0  &  --  &  3.88\mbox{e-}0  &  --  &  6.18\mbox{e-}0  &  --  \\
 & 1  &  2.15\mbox{e-}1  &  0.85  &  1.19\mbox{e-}0  &  0.69  &  2.61\mbox{e-}0  &  0.57  &  3.88\mbox{e-}0  &  0.67  \\
1 & 2  &  1.12\mbox{e-}1  &  0.94  &  6.60\mbox{e-}1  &  0.85  &  1.54\mbox{e-}0  &  0.76  &  2.35\mbox{e-}0  &  0.72  \\
 & 3  &  5.73\mbox{e-}2  &  0.97  &  3.47\mbox{e-}1  &  0.93  &  8.45\mbox{e-}1  &  0.87  &  1.31\mbox{e-}0  &  0.85  \\
 & 4  &  2.89\mbox{e-}2  &  0.99  &  1.78\mbox{e-}1  &  0.96  &  4.44\mbox{e-}1  &  0.93  &  6.89\mbox{e-}1  &  0.92  \\
\hline
 & 0  &  1.75\mbox{e-}2  &  --  &  3.33\mbox{e-}1  &  --  &  1.19\mbox{e-}0  &  --  &  2.18\mbox{e-}0  &  --  \\
 & 1  &  2.19\mbox{e-}3  &  3.00  &  4.28\mbox{e-}2  &  2.96  &  1.64\mbox{e-}1  &  2.86  &  3.39\mbox{e-}1  &  2.68  \\
2 & 2  &  2.69\mbox{e-}4  &  3.03  &  5.08\mbox{e-}3  &  3.08  &  1.88\mbox{e-}2  &  3.12  &  3.83\mbox{e-}2  &  3.15  \\
 & 3  &  3.32\mbox{e-}5  &  3.02  &  6.12\mbox{e-}4  &  3.05  &  2.23\mbox{e-}3  &  3.08  &  4.44\mbox{e-}3  &  3.11  \\
 & 4  &  4.12\mbox{e-}6  &  3.01  &  7.50\mbox{e-}5  &  3.03  &  2.70\mbox{e-}4  &  3.04  &  5.32\mbox{e-}4  &  3.06  \\
\hline
& 0  &  2.78\mbox{e-}4  &  --  &  8.57\mbox{e-}3  &  --  &  8.97\mbox{e-}2  &  --  &  2.19\mbox{e-}1  &  --  \\
 & 1  &  8.71\mbox{e-}6  &  4.99  &  2.52\mbox{e-}4  &  5.09  &  2.52\mbox{e-}3  &  5.15  &  5.67\mbox{e-}3  &  5.27  \\
3 & 2  &  2.70\mbox{e-}7  &  5.01  &  7.54\mbox{e-}6  &  5.06  &  7.36\mbox{e-}5  &  5.10  &  1.60\mbox{e-}4  &  5.15  \\
 & 3  &  8.36\mbox{e-}9  &  5.01  &  2.30\mbox{e-}7  &  5.04  &  2.21\mbox{e-}6  &  5.06  &  4.75\mbox{e-}6  &  5.08  \\
 & 4  &  2.6\mbox{e-}10  &  5.01  &  7.08\mbox{e-}9  &  5.02  &  6.77\mbox{e-}8  &  5.03  &  1.44\mbox{e-}7  &  5.04  \\
\hline
 \end{array} $
}
\end{center}{$\phantom{|}$}
     \caption{Convergence  of the approximate eigenvalues $\lambda_h$ for (the mixed-degree) Case 2 with $\tau = 1$. % The polynomial degree $k$ is used for $u_h$ and $\widehat{u}_h$, while $k-1$ for $\vec{q}_h$.
     }
   \label{tab4d}
\end{table}
\begin{table}%[thbp]
  \begin{center}
  \scalebox{0.9}{%
    $\begin{array}{|c|c||c c  | c c|  c c|  c c|}
    \hline
  \mbox{degree} & \mbox{mesh} & \multicolumn{2}{|c}{\mbox{first mode}} &  \multicolumn{2}{|c}{\mbox{second mode}} & \multicolumn{2}{|c|}{\mbox{fourth mode}} & \multicolumn{2}{|c|}{\mbox{sixth mode}} \\
    k & \ell & \mbox{error}  & \mbox{order} & \mbox{error}  & \mbox{order} & \mbox{error}  & \mbox{order} & \mbox{error}  & \mbox{order} \\
 \hline
& 0  &  3.77\mbox{e-}1  &  --  &  1.87\mbox{e-}0  &  --  &  3.80\mbox{e-}0  &  --  &  6.24\mbox{e-}0  &  --  \\
 & 1  &  2.13\mbox{e-}1  &  0.83  &  1.17\mbox{e-}0  &  0.68  &  2.56\mbox{e-}0  &  0.57  &  3.83\mbox{e-}0  &  0.70  \\
1 & 2  &  1.12\mbox{e-}1  &  0.93  &  6.55\mbox{e-}1  &  0.83  &  1.53\mbox{e-}0  &  0.74  &  2.33\mbox{e-}0  &  0.72  \\
 & 3  &  5.72\mbox{e-}2  &  0.97  &  3.46\mbox{e-}1  &  0.92  &  8.42\mbox{e-}1  &  0.86  &  1.30\mbox{e-}0  &  0.84  \\
 & 4  &  2.89\mbox{e-}2  &  0.99  &  1.78\mbox{e-}1  &  0.96  &  4.43\mbox{e-}1  &  0.93  &  6.89\mbox{e-}1  &  0.92  \\
\hline
 & 0  &  2.26\mbox{e-}4  &  --  &  1.95\mbox{e-}3  &  --  &  1.01\mbox{e-}2  &  --  &  2.36\mbox{e-}2  &  --  \\
 & 1  &  3.18\mbox{e-}5  &  2.83  &  5.02\mbox{e-}4  &  1.96  &  1.66\mbox{e-}3  &  2.61  &  3.05\mbox{e-}3  &  2.95  \\
2 & 2  &  2.58\mbox{e-}6  &  3.62  &  4.48\mbox{e-}5  &  3.48  &  1.34\mbox{e-}4  &  3.63  &  2.15\mbox{e-}4  &  3.83  \\
 & 3  &  1.80\mbox{e-}7  &  3.84  &  3.20\mbox{e-}6  &  3.81  &  1.08\mbox{e-}5  &  3.62  &  1.52\mbox{e-}5  &  3.82  \\
 & 4  &  1.18\mbox{e-}8  &  3.93  &  2.12\mbox{e-}7  &  3.92  &  7.40\mbox{e-}7  &  3.87  &  1.01\mbox{e-}6  &  3.92  \\
\hline
& 0  &  4.92\mbox{e-}6  &  --  &  1.20\mbox{e-}4  &  --  &  3.98\mbox{e-}3  &  --  &  1.11\mbox{e-}2  &  --  \\
 & 1  &  4.93\mbox{e-}8  &  6.64  &  7.34\mbox{e-}7  &  7.35  &  1.77\mbox{e-}5  &  7.82  &  7.32\mbox{e-}5  &  7.25  \\
3 & 2  &  5.53\mbox{e-}10  &  6.48  &  5.59\mbox{e-}9  &  7.04  &  1.44\mbox{e-}7  &  6.94  &  8.61\mbox{e-}7  &  6.41  \\
 & 3  &  5.63\mbox{e-}12  &  6.62  &  4.44\mbox{e-}11  &  6.97  &  1.69\mbox{e-}9  &  6.41  &  1.25\mbox{e-}8  &  6.11  \\
 & 4  &  1.04\mbox{e-}13  &  5.76  &  4.01\mbox{e-}13  &  6.79  &  2.40\mbox{e-}11  &  6.14  &  1.71\mbox{e-}10  &  6.19  \\
\hline
 \end{array} $
}
\end{center}{$\phantom{|}$}
     \caption{Convergence  of the postprocessed eigenvalues $\lambda_h^\ast$ for (the mixed-degree) Case 2 with $\tau = 1$. % The polynomial degree $k$ is used for $u_h$ and $\widehat{u}_h$, while $k-1$ for $\vec{q}_h$
       }
   \label{tab4e}
\end{table}
\begin{table}%[thbp]
  \begin{center}
  \scalebox{0.9}{%
    $\begin{array}{|c|c||c c  | c c|  c c|  c c|}
    \hline
  \mbox{degree} & \mbox{mesh} & \multicolumn{2}{|c}{\mbox{first mode}} &  \multicolumn{2}{|c}{\mbox{second mode}} & \multicolumn{2}{|c|}{\mbox{fourth mode}} & \multicolumn{2}{|c|}{\mbox{sixth mode}} \\
    k & \ell & \mbox{error}  & \mbox{order} & \mbox{error}  & \mbox{order} & \mbox{error}  & \mbox{order} & \mbox{error}  & \mbox{order} \\
\hline
& 0  &  8.30\mbox{e-}4  &  --  &  2.40\mbox{e-}2  &  --  &  1.35\mbox{e-}1  &  --  &  2.32\mbox{e-}1  &  --  \\
 & 1  &  3.33\mbox{e-}5  &  4.64  &  1.17\mbox{e-}3  &  4.35  &  9.96\mbox{e-}3  &  3.76  &  1.47\mbox{e-}2  &  3.98  \\
1 & 2  &  1.65\mbox{e-}6  &  4.33  &  6.63\mbox{e-}5  &  4.15  &  6.43\mbox{e-}4  &  3.95  &  9.44\mbox{e-}4  &  3.96  \\
 & 3  &  9.31\mbox{e-}8  &  4.15  &  3.98\mbox{e-}6  &  4.06  &  4.02\mbox{e-}5  &  4.00  &  5.92\mbox{e-}5  &  4.00  \\
\hline
 & 0  &  6.54\mbox{e-}5  &  --  &  1.37\mbox{e-}3  &  --  &  1.33\mbox{e-}2  &  --  &  2.03\mbox{e-}2  &  --  \\
 & 1  &  1.12\mbox{e-}6  &  5.86  &  2.62\mbox{e-}5  &  5.71  &  2.71\mbox{e-}4  &  5.62  &  4.76\mbox{e-}4  &  5.41  \\
2 & 2  &  1.80\mbox{e-}8  &  5.97  &  4.31\mbox{e-}7  &  5.93  &  4.53\mbox{e-}6  &  5.90  &  8.60\mbox{e-}6  &  5.79  \\
 & 3  &  2.82\mbox{e-}10  &  5.99  &  6.82\mbox{e-}9  &  5.98  &  7.21\mbox{e-}8  &  5.98  &  1.39\mbox{e-}7  &  5.95  \\
\hline
 & 0  &  5.73\mbox{e-}7  &  --  &  1.97\mbox{e-}5  &  --  &  5.21\mbox{e-}4  &  --  &  7.91\mbox{e-}4  &  --  \\
 & 1  &  2.32\mbox{e-}9  &  7.95  &  2.30\mbox{e-}7  &  6.42  &  2.33\mbox{e-}6  &  7.80  &  3.66\mbox{e-}6  &  7.75  \\
3 & 2  &  7.82\mbox{e-}12  &  8.21  &  9.08\mbox{e-}10  &  7.99  &  9.36\mbox{e-}9  &  7.96  &  1.47\mbox{e-}8  &  7.96  \\
 & 3  &  3.12\mbox{e-}14  &  7.97  &  1.89\mbox{e-}12  &  8.91  &  3.15\mbox{e-}11  &  8.22  &  5.24\mbox{e-}11  &  8.13  \\
\hline
 \end{array} $
}
\end{center}{$\phantom{|}$}
     \caption{{Convergence  of the postprocessed BDM eigenvalues $\lambda_h^\ast$ obtained using the implementation of  mixed-degree Case 1 with $\tau = 0$.% The polynomial degree $k$ is used for $\vec{q}_h$ and $\widehat{u}_h$, while $k-1$ for $u_h$.
       }}
   \label{tab4f}
\end{table}

}}

% Suggestions on things to do for this section:

% \begin{enumerate}

% \item Give an illustration of Theorem~\ref{thm:ewrate}, $O(h^{2k+1})$
%   convergence of eigenvalues, using an example with infinitely smooth
%   eigenfunctions (like on a square).

% \item Study and tabulate or graphically visualize the dependence of
%   the eigenvalue errors with $\tau$.

% \item Describe the postprocessing as a numerical technique (with no
%   theory) to increase the order to $O(h^{2k+2})$ and give results on
%   its numerical performance - can mention that we do not have a
%   theoretical explanation for this.

% \item Give an illustration of Theorem~\ref{thm:init} by computing and
%   tabulating $|\lambda_h-\lt_h|$ for the first few eigenvalues as a
%   function of $h$.  In the RT-case~\cite{hrt:2010} this difference was
%   $O(h^2)$, not $O(h)$ as stated in Theorem~\ref{thm:init} -- I am
%   curious to see if this discrepancy between HDG and HRT shows up in
%   practice.

% \item Optional: Any low regularity example.

% \end{enumerate}

% \bibliographystyle{siam}
% \bibliography{/Users/Jay/articles/bib/lib}

\end{document}